\documentclass[a4paper, 12pt, final]{amsart}

\makeatletter
\numberwithin{equation}{section}
\numberwithin{figure}{section}

\usepackage{amsmath,amsfonts,amssymb,amsthm,epsfig}

\usepackage{color}
\usepackage[final,linkcolor = blue,citecolor = blue,colorlinks=true]{hyperref}
\usepackage[leqno]{amsmath}

\usepackage{amscd}
\usepackage{mathrsfs}
\usepackage{}
\usepackage[all]{xy}
\usepackage[OT1]{fontenc}
\usepackage[latin1]{inputenc}
\usepackage{amssymb,latexsym}
\usepackage{type1cm,ae}
\usepackage[notref,notcite]{showkeys}
\usepackage{graphicx}
\usepackage{xspace}
\usepackage{setspace}
\usepackage[english]{babel}
\usepackage{comment}
\usepackage{color}

\theoremstyle{definition}
		\newtheorem{theorem}{Theorem}[section]
				\newtheorem{proposition}[theorem]{Proposition}
				\newtheorem{lemma}[theorem]{Lemma}
				
				\newtheorem{corollary}[theorem]{Corollary}

     	        \newtheorem{definition}[theorem]{Definition}
	            \newtheorem{remark}[theorem]{Remark}
\numberwithin{equation}{section}

\newcommand*{\bR}{\ensuremath{\mathbb{R}}}

\newcommand*{\Wert}{\mathord{\mbox{|\kern-1.5pt|\kern-1.5pt|}}}
\newcommand*{\ie}{\mbox{i.e.}\xspace}

\newcommand\numberthis{\addtocounter{equation}{1}\tag{\theequation}}

\newcommand{\N}{\mathbb{N}}	
\newcommand{\R}{\mathbb{R}}	

\newcommand{\HH}{\mathbb{H}}	
\newcommand{\G}{\mathbb{G}}	
\newcommand{\Co}{\mathscr{C}}	

\renewcommand{\Vec}{\mathrm{Vec}}	
\newcommand{\Span}{\mathrm{span}}	

\newcommand{\dd}{\,\mathrm{d}}	
\newcommand{\ph}[1]{\phantom{#1}}	
\newcommand{\ol}[1]{\overline{#1}}
\renewcommand{\a}{\alpha}
\renewcommand{\b}{\beta}

\newcommand{\mD}{\mathcal{D}}
\renewcommand{\ph}{\phantom{\alpha}}

\DeclareMathOperator{\Modd}{Mod}

\DeclareMathOperator{\modulus}{Mod}

\DeclareMathOperator{\Lip}{Lip}

\DeclareMathOperator{\Vol}{Vol}
\DeclareMathOperator{\card}{card}
\DeclareMathOperator{\Dett}{Det}

\newcommand{\scr}[1]{\mathscr{#1}}
\newcommand{\frk}[1]{\mathfrak{#1}}

\newcommand{\g}{\frk g}

\def\Xint#1{\mathchoice
  {\XXint\displaystyle\textstyle{#1}}%
  {\XXint\textstyle\scriptstyle{#1}}%
  {\XXint\scriptstyle\scriptscriptstyle{#1}}%
  {\XXint\scriptscriptstyle\scriptscriptstyle{#1}}%
  \!\int}
\def\XXint#1#2#3{{\setbox0=\hbox{$#1{#2#3}{\int}$}
  \vcenter{\hbox{$#2#3$}}\kern-.5\wd0}}

\def\dashint{\Xint-}

\dedicatory{Dedicated to Seppo Rickman and Jussi V\"ais\"al\"a on the occasion of their $80^\text{th}$ birthdays}
\makeatother

\usepackage{babel}
\begin{document}

\title[Equivalence of quasiregularity via the Popp extension]{Equivalence of quasiregular mappings on subRiemannian manifolds via the Popp extension}
\date{\today}
\author[Chang-Yu Guo]{Chang-Yu Guo}
\address[Chang-Yu Guo]{Department of Mathematics and Statistics, University of Jyv\"askyl\"a and Department of Mathematics, University of Fribourg}
\email{changyu.guo@unifr.ch}

\author[Tony Liimatainen]{Tony Liimatainen}
\address[Tony Liimatainen]{Department of Mathematics and Statistics, University of Jyv\"askyl\"a}
\email{tony.t.liimatainen@jyu.fi}


\begin{abstract}
We show that all the common definitions of quasiregular mappings $f\colon M\to N$ between two equiregular subRiemannian manifolds of homogeneous dimension $Q\geq 2$ are quantitatively equivalent with precise dependences of the quasiregularity constants. As an immediate consequence, we obtain that if $f$ is $1$-quasiregular according to one of the definitions, then it is also $1$-quasiregular according to any other definition. In particular, this recovers a recent theorem of Capogna et al.~\cite{clo14} on the equivalence of $1$-quasiconformal mappings.

Our main results answer affirmatively a few open questions from the recent research. The main new ingredient in our proofs is the distortion estimates for particular local extensions of the horizontal metrics. These extensions are named ``Popp extensions", and based on these extensions, we introduce a new natural and invariant definition of quasiregularity in the equiregular subRiemannian setting. The analysis on Popp extensions and on the implied distortion is also of independent interest. 
\end{abstract}

\maketitle
\tableofcontents{}

\section{Introduction}\label{sec:introduction}
A continuous mapping $f\colon X\to Y$ between topological spaces is said to be a \textit{branched covering} if $f$ is both \textit{discrete} and \textit{open}. Recall that a mapping $f\colon X\to Y$ between two topological spaces is discrete if the preimage $f^{-1}(y)$ of each point $y\in Y$ is a discrete subset of $X$. 

For a branched covering $f\colon X\to Y$ between two metric spaces $(X,d)$ and $(Y,d)$, $x\in X$ and $r>0$, we set
\begin{equation}\label{pw_distort}
H_f(x,r):=\frac{L_f(x,r)}{l_f(x,r)},
\end{equation}
where 
\begin{equation*}
L_f(x,r):=\sup_{y\in X}{\{d(f(x),f(y)):d(x,y)= r\}},
\end{equation*}
and
\begin{equation*}
l_f(x,r):=\inf_{y\in X}{\{d(f(x),f(y)):d(x,y)= r\}}.
\end{equation*}
Then the \textit{linear dilatation function} $H_f$ of $f$ is defined pointwise by 
\begin{equation*}
H_f(x)=\limsup_{r\to0}H_f(x,r), \quad x\in X.
\end{equation*}

\begin{definition}\label{def:metric quasiregular map}
	A branched covering $f\colon X\to Y$ between two metric measure spaces is termed \textit{metrically $H$-quasiregular} if the linear dilatation function $H_f$ is finite everywhere and essentially bounded from above by finite constant $H\in [1,\infty)$. 
\end{definition}

If $f\colon X\to Y$ in Definition~\ref{def:metric quasiregular map} is additionally assumed to be homeomorphic, then $f$ is called \textit{metrically $H$-quasiconformal\footnote{The definition of metric quasiconformality is somewhat weaker than the more commonly used definition in literature. Namely, we do not require that the linear dilatation is bounded everywhere. However, the two definitions coincide if the metric measure spaces have locally bounded geometry.}}. We will call $f$ in short a metrically quasiregular (quasiconformal) mapping if it is metrically $H$-quasiregular (-quasiconformal) for some $H\in [1,\infty)$. 

The importance of quasiconformal mappings in
complex analysis was realized by Ahlfors, Teichm\"uller and Morrey
in the 1930s. Ahlfors used quasiconformal
mappings in his geometric approach to Nevanlinna's
value distribution theory that earned him one of the first two Fields
medals. Teichm\"uller used quasiconformal mappings
to measure the distance between two conformally inequivalent
compact Riemann surfaces, starting what
is now called Teichm\"uller theory. Morrey proved a powerful existence theorem, called the 
measurable Riemann mapping theorem, which has had tremendous impact on complex analysis and
dynamics, Teichm\"uller theory, low dimensional
topology, inverse problems and partial differential equations.   

The higher-dimensional theory of quasiconformal mappings was
initiated in earnest by Reshetnyak, Gehring and
V\"ais\"al\"a in the early 1960s~\cite{g62,g63,re89}. Generalizations to quasiregular mappings were introduced in the works by Reshetnyak, 
and the basic theory was comprehensively laid and significantly advanced
in a sequence of papers from the Finnish school of Martio, Rickman and
V\"ais\"al\"a in the late 1960s~\cite{mrv69,mrv70,mrv71}.

Historically, there are three different definitions of quasiconformality: the metric definition, the analytic definition, and the geometric definition (see Section~\ref{subsec:Definitions of quasiregular mappings} for a detailed description of these definitions). Apparently, the metric definition is of infinitesimal flavor, the analytic definition is a pointwise condition, and the geometric definition is more of a global nature. It is a rather deep fact, due to Gehring~\cite{g63} and V\"ais\"al\"a~\cite{v61}, that all the three definitions of quasiregularity are equivalent, quantitatively. The interplay of all three aspects of quasiconformality/quasiregularity is an
important feature of the theory; see~\cite{bi83,im01,re89,r93,v71} for more on the Euclidean theory of these mappings. 

The interplay between the different definitions of quasiconformality and quasiregularity in more general spaces is quite intricate and has drawn attention from many mathematicians working in the analysis on metric spaces. In the setting of Ahlfors regular spaces supporting an abstract Poincar\'e inequality, the foundational results were established by Heinonen and Koskela~\cite{hk98}. More delicate results were obtained later by Tyson in a sequence of papers~\cite{t98,t00,t01}. In particular, the equivalence of the definitions were achieved in the setting of metric spaces of locally bounded geometry in the seminal work~\cite{hkst01}. Among others, let us point out that, by the surprising work of Williams~\cite{w12proc}, the analytic quasiconformality is always equivalent to the geometric quasiconformality without any a priori assumption on the underlying metric measure spaces. See also~\cite{w14} for the most recent results along this direction.

The question of the equivalences in the quasiregular category arose naturally after the fundamental work of Heinonen and Rickman~\cite{hr02}, where a metric theory of branched coverings between \textit{generalized manifolds} were developed. Some of the equivalences were later obtained by Onninen and Rajala~\cite{or09} in the setting of branched coverings from Euclidean spaces to generalized manifolds of certain type. The most recent remarkable advance is due to Williams~\cite{w12}, where these equivalences were obtained in metric spaces of locally bounded geometry; see also the work of Cristea~\cite{c06}.

In this paper, we consider these equivalences on \textit{subRiemannian manifolds}. Our study is partially motived by the study of uniformly quasiregular mappings in the subRiemannian manifolds~\cite{bfp12,mp15,flp14}. On the other hand, together with the very recent paper~\cite{gnw15}, it completes the foundations of the theory of quasiregular mappings in general in the setting of \textit{equiregular} subRiemannian manifolds. We refer to~\cite{hh97,d99,d00,flp14,gnw15} and the references therein for more on the theory of quasiregular mappings on subRiemannian manifolds, and also on the historical development of the theory.

Let $(M,g)$ be an equiregular subRiemannian $n$-manifold of homogeneous dimension $Q\geq 2$. We use the notation $d_M$ and $\Vol_M$  for the corresponding \textit{subRiemannian distance} and \textit{Popp volume} on $(M,g)$ respectively (precise definitions are given in Section~\ref{sec:Geometry of equiregular subRiemannian manifolds} below). Alternatively, we use $P_g$ to denote the Popp volume form on $(M,g)$. If not stated otherwise, we use the Popp measure as the measure on the subRiemannian manifolds in question throughout the text. 

Our main result of this paper is the following precise quantitative equivalence result on various definitions of quasiregularity in the equiregular subRiemannian setting. We express the relation between the different quasiregularity constants by using constants $H,\widehat{H}, K$ and $\widehat{K}$, and then explain how these are related. We remark at this point, in a slightly vague manner, that the constants $H$ and $\widehat{H}$ bound ``horizontal distortion'', while $K$ and $\widehat{K}$ are bounds for ``total distortion''. 

\begin{theorem}\label{thm:equivalence of quasiregular mappings}
	Let $f\colon (M,g)\to (N,h)$ be a branched covering between two equiregular subRiemannian manifolds of homogeneous dimension $Q\geq 2$ and rank $k$. Then the following conditions are quantitatively equivalent:
	\begin{align*}
	1)\quad  f  &\mbox{ is a metrically $H$-quasiregular mapping}, \\
	2)\quad f  &\mbox{ is a weak metrically $H$-quasiregular mapping}, \\ 
	3)\quad f &\mbox{ is a horizontally $\widehat{H}$-quasiregular mapping}, \\ 
	4)\quad f  &\mbox{ is an analytically $K$-quasiregular mapping}, \\ 
	5)\quad f &\mbox{ is a geometrically $K$-quasiregular mapping }, \\
	6)\quad f  &\mbox{ is a subRiemannian $\widehat{K}$-quasiregular mapping}. 
	\end{align*}
	We have the following precise dependences on the quasiregularity constants $H,\widehat{H}, K$ and $\widehat{K}$:
	\begin{itemize}
	 \item If $f$ is weak metrically $H$-quasiregular, then it is analytically $K$-quasiregular with $K=H^{Q-1}$ and horizontally $\widehat{H}$-quasiregular with $\widehat{H}=H^{k-1}$. 
	
	\item If $f$ is analytically $K$-quasiregular, then it is metrically $H$-quasiregular with $H=K$ and horizontally $\widehat{H}$-quasiregular with $\widehat{H}=K$. 
	
	\item If $f$ is horizontally $\widehat{H}$-quasiregular, then $f$ is analytically $K$-quasiregular with $K=\widehat{H}^{Q-1}$, metrically $H$-quasiregular with $H=\widehat{H}$, and subRiemannian $\widehat{K}$-quasiregular with $\widehat{K}=\widehat{H}^{Q-1}$.
	
	\item If $f$ is subRiemannian $\widehat{K}$-quasiregular, then $f$ is horizontally $K\widehat{H}$-quasiregular with $\widehat{H}=\widehat{K}$.
	\end{itemize}
\end{theorem}

Theorem~\ref{thm:equivalence of quasiregular mappings} was predicted to be true in~\cite[below Theorem 5.8]{gnw15} and it answers affirmatively an open question of F\"assler--Lukyanenko--Peltonen~\cite[Question 2.18]{flp14}. As we will see in a moment in Section~\ref{subsec:A remark on the definitions of quasiregularity in the Heisenberg groups}, as a byproduct of our proof, it also answers affirmatively~\cite[Question 2.17]{flp14}.

In the special case when $f$ is a homeomorphism between two Riemannian $n$-manifolds $(M,g)$ and $(N,h)$, the equivalences of different notions of quasiregularity showing in Theorem~\ref{thm:equivalence of quasiregular mappings} are well-known and follows directly from~\cite[Theorem 9.8]{hkst01} and~\cite[Theorem 5.2]{bkr07}. But these works do not imply such precise dependences of the quasiregularity constants as claimed in Theorem~\ref{thm:equivalence of quasiregular mappings}. In particular, it was not even clear whether all the definitions of $1$-quasiconformality are equivalent even in the Riemannian setting before the very recent work~\cite{clo14} (2016). Thus, the precise interrelation of the quasiregular constants obtained in Theorem~\ref{thm:equivalence of quasiregular mappings} seems to be new even in the setting of Riemannian manifolds. 

As an immediate consequence of Theorem~\ref{thm:equivalence of quasiregular mappings}, we obtain the following characterization of $1$-quasiregularity (or conformality).

\begin{corollary}\label{coro:equivalence of 1-quasiregular mappings}
Let $f\colon (M,g)\to (N,h)$ be a branched covering between two equiregular subRiemannian manifolds of homogeneous dimension $Q\geq 2$. Then the following conditions are equivalent:
\begin{align*}
	1)\quad  f  &\mbox{ is a metrically $1$-quasiregular mapping}, \\
	2)\quad  f  &\mbox{ is a weak metrically $1$-quasiregular mapping}, \\ 
	3)\quad f &\mbox{ is a horizontally $1$-quasiregular mapping}, \\ 
	4)\quad f  &\mbox{ is an analytically $1$-quasiregular mapping}, \\ 
	5)\quad f &\mbox{ is a geometrically $1$-quasiregular mapping }, \\
	6)\quad f  &\mbox{ is a subRiemannian $1$-quasiregular mapping}. 
	\end{align*}
Moreover, each of the conditions above is further equivalent to the following equation of \emph{horizontal conformality}
\begin{align}\label{eq:equation of conformality}
	f^*h=cg.
\end{align}
Where $c$ is some a.e. positive function on $M$ and $f$ satisfied the regularity of assumptions of some (and thus all) definitions of $1$-quasiregularity in the theorem above. 

\end{corollary}

Taking into account the geometry of subRiemannian manifolds, the definitions of subRiemannian quasiregularity and horizontal quasiregularity that appear in Theorem~\ref{thm:equivalence of quasiregular mappings} and in Corollary~\ref{coro:equivalence of 1-quasiregular mappings} seem to be quite natural in the setting of equiregular subRiemannian manifolds. The motivation for the first of these definitions arises from the corresponding one in the setting of Riemannian manifolds~\cite{l14} and also from the one in the Heisenberg group~\cite{d99,d00}. The definitions of subRiemannian and horizontal quasiregularity are given in Section~\ref{sec:Equivalence of the horizontal and the subRiemannian quasiregularity}.

To define the notion of subRiemannian quasiregularity, we need local extensions of the horizontal metrics. We call these local extensions \textit{Popp extensions}. They are introduced in Section~\ref{sec:The Popp extension of a horizontal metric}.  Regarding the proof of Theorem~\ref{thm:equivalence of quasiregular mappings}, beside of subRiemannian calculus developed in the recent papers~\cite{mm95,flp14,gnw15}, the Popp extensions turns out to be the key in the proofs. In particular, it allows us to recover the metric quasiregularity from the subRiemannian one.

When $f\colon (M,g)\to (N,h)$ is a homeomorphism, the equivalences of $1)$, $2)$, $3)$ and $6)$ in Corollary~\ref{coro:equivalence of 1-quasiregular mappings} reduces to a recent result in the mentioned paper of Capogna et al.~\cite[Theorem 1.3]{clo14} (2016), where remarkably $1$-quasiconformal mappings were shown to be smooth when the underlying subRiemannian manifolds support suitable regularity theory for subelliptic $p$-Laplacian operators. The latter result, when $M$ and $N$ are Carnot groups, is due to Capogna and Cowling~\cite{cc06}.

In this paper, we did not consider Theorem~\ref{thm:equivalence of quasiregular mappings} in the setting of non-equiregular subRiemannian manifolds. It would be interesting to know whether a version of Theorem~\ref{thm:equivalence of quasiregular mappings} holds in such manifolds; see~\cite{a15} for more discussion along this direction.

This paper is organized as follows. After this introductory section, 
Section~\ref{sec:Geometry of equiregular subRiemannian manifolds} contains the basic geometry of equiregular subRiemannian manifolds. We review the basic calculus in subRiemannian manifolds in Section~\ref{sec:Calculus on subRiemannian manifolds}. In Section~\ref{sec:Quasiregular mappings}, we consider various definitions of quasiregularity in the subRiemannian manifolds and collect some basic properties of quasiregular mappings. In Section~\ref{sec:The Popp extension of a horizontal metric}, we introduce Popp extensions of horizontal, i.e. subRiemannian, metrics. There we also introduce the subRiemannian quasiregularity and horizontal quasiregularity. We give a proof of Theorem~\ref{thm:equivalence of quasiregular mappings} in Section~\ref{sec:Proofs of the main result}. For the convenience of readers who are not familiar with the subRiemannian geometry, we also included in the appendix the proof of the equivalences of $1$-quasiregularity for mappings between Riemannian manifolds.

\section{Geometry of equiregular subRiemannian manifolds}\label{sec:Geometry of equiregular subRiemannian manifolds}
\newcommand{\Dst}{\mathcal{D}}

\subsection{Basic concepts in subRiemannian geometry}
Let $M$ be a $C^\infty$-smooth manifold of dimension $n$ and let $\Dst\subset TM$ be a subbundle of constant rank $k$.
Define the following \emph{flag of distributions} inductively for $s\in\N$:
\[
\begin{cases}
 	\Dst^{0} &:= \{0\} \\
	\Dst^{1} &:= \Gamma(\Dst) \\
	\Dst^{s+1} &:= \Dst^{s} + \Co^\infty(M)\text{-}\Span\left\{[X,Z] : X\in\Dst^{1},\ Z\in\Dst^{s} \right\} ,
\end{cases}
\]
where $\Gamma(\Dst)$ is the set of all smooth sections of $\Dst$. For any set $S$ of vector fields, $\Co^\infty(M)\text{-}\Span(S)$ is the set of linear combinations of elements of $S$ with coefficients in the ring $\Co^\infty(M)$ of smooth functions $M\to\R$.

By definition we have
\[
\{0\}\subset \Dst^{1} \subset\dots\subset\Dst^{s}\subset\Dst^{s+1}\subset\dots\subset\Vec(M), 
\]
where $\Vec(M)$ is the space of all vector fields on $M$.
For any point $p\in M$, we have a \emph{pointwise flag} 
\[
\Dst^{s}_p := \{ Z(p):Z\in\Dst^{s}\} \subset T_pM .
\]

To such a flag we associate the following functions $M\to\N\cup\{+\infty\}$:
\begin{description}
\item[Ranks] 	
	$k_s(p) := \dim(\Dst_p^{s})$, $s=1,2,\ldots$. 
\item[Growth vector] 
	$n_s(p) := k_s(p)-k_{s-1}(p) = \dim( \Dst_p^{s}/\scr \Dst_p^{s-1} )$. 
	
\item[Step]	
	$m(p) := \inf\{ s: \Dst_p^{s} = T_pM\}$. 

\item[Weight]	
	for $i\in\{1,\dots,n\}$: $w_i:=s$ if and only if $i\in\{k_{s-1}+1,\dots,k_s\}$. 

\end{description} 
We have that that $k=k_1\le k_2\le\dots\le n$ and also $\sum_{i=1}^s n_i = k_s$. The function $p\mapsto (n_1(p),n_2(p),\dots)\in\N^\N$ is usually called the \emph{growth vector}. Notice that if $m(p)<\infty$, then 
	\[
	\{0\}\subset\Dst_p^{1}\subset\dots\subset\Dst_p^{m(p)}=T_pM .
	\]
The subbundle $\Dst$ is usually called the \emph{horizontal distribution} and it is said to be \emph{equiregular} if $k_s$, and hence $n_s$ and $m$, are constants.
If $m<\infty$, then $\Dst$ is said to be \emph{bracket generating} and we have $k_s(p)=k_s=n$.
\begin{definition}[subRiemannian manifold]
 	An \emph{equiregular subRiemannian manifold} is a triple $(M,\Dst,g)$ where $M$ is a smooth and connected manifold, $\Dst\subset TM$ is a bracket generating equiregular subbundle, and $g$ is a smooth inner product on the fibers $\mD_p$, $p\in M$, of $\mD$. 
\end{definition}

The inner product $g$ is called a \emph{horizontal metric} of $\Dst$. We use the notation $|v|_{g_p}$ or $\|v\|_{g_p}$ for the norm $\sqrt{g_p(v,v)}$ of a horizontal vector $v\in \Dst_p$. 
When there is no risk of confusion, we sometimes remove the subscript and write simply $|v|$ or $|v|_g$ etc.

\begin{definition}[subRiemannian distance]
	An absolutely continuous curve 
	$\gamma\colon [0,1]\to M$ is called a \emph{horizontal curve} if $\gamma'(t)\in \Dst_{\gamma(t)}$ for almost every $t\in[0,1]$. 
	
	The \emph{length} $l(\gamma)$ of a horizontal curve $\gamma\colon [0,1]\to M$ is
	\[
	l(\gamma) := \int_0^1\|\gamma'(t)\|\dd t.
	\]
	The \emph{subRiemannian distance} is defined by:
	\[
	d_g(p,q) := \inf_\gamma\left\{
	l(\gamma):
	\text{ $\gamma$ is a horizontal curve joining $p\in M $ to $q\in M$}
	\right\} .
	\]
\end{definition}

A subRiemannian manifold $M$ can be endowed in a canonical way with a smooth volume $\Vol_M$ that is called the \emph{Popp measure}.
The construction can be found in Section~\ref{sec:The Popp extension of a horizontal metric} below; see also \cite{br13,gj13,m02}. Consequently, when a subRiemannian manifold $M$ is endowed with the subRiemannian distance $d_g$ and the Popp measure $\Vol_M$, then $(M,d_g,\Vol_M)$ becomes a metric measure space. We will often drop the subscript $g$ from the distance $d_g$ when the context is clear.

\subsection{Tangent cone of an equiregular subRiemannian manifold}\label{subsec:tangent cone}
To introduce the notion of $P$-differentiability of a mapping between two equiregular subRiemannian manifolds, we need the notion of privileged coordinates and the construction of the metric tangent cone of equiregular subRiemannian manifolds. We briefly recall these nowadays well-known concepts and refer to literature for more details. Below $(M,\mD)$ is an equiregular subRiemannian manifold. 

\begin{definition}[Non-holonomic order]
	Let $f\colon M\to \R$ be a smooth function and $o\in M$.
	The \emph{non-holonomic order of $f$ at $p$} is defined as the maximum of $k\in\N$ such that for all $i<k$ and for any choice of horizontal vector fields $X_1,\dots,X_i\in\Dst$ it holds that
	\[
	X_1X_2\cdots X_if(p) = 0 .
	\]
\end{definition}
\begin{definition}
	Let $p\in M$ and let $U$ be an open neighborhood of $p$ in $M$.
	A system of coordinates $(x_1,\dots,x_n)\colon U\to \R^n$ centered at $p$ is a \emph{system of privileged coordinates} if each of the functions $x_i$ has non-holonomic order equaling the weight $w_i$.
\end{definition}
Privileged coordinates exists at all points of $M$. See \cite{b96,m02,cr15} for insights in this argument.

Let $(M,\mathcal{D},d)$ be a subRiemannian manifold with horizontal distribution $\Dst\subset TM$.
Since the results of this section we discuss are local, we can assume that $\Dst$ is generated by $k$ smooth vector fields $X_1,\dots,X_k\in\Gamma(TM)$ that are linearly independent at every point.

We next briefly discuss the construction of the metric tangent cone at a point $o\in M$ following the presentation in~\cite{gnw15}. See also \cite{b96,Gromov96,flp14} for more information about the construction of a tangent cone in the subRiemannian setting.

We assume that at each point $p\in M$ a system of privileged coordinates is chosen.
Let $\delta_\epsilon^p\colon U^p_\epsilon\to U^p_{\epsilon^{-1}}$ be the corresponding dilations , $\epsilon\in(0,+\infty)$, where $U^p_\epsilon\subset M$ is an open neighborhood of $p$. 
We assume $U^p_1\subset U^p_\epsilon$ for all $\epsilon\le 1$. Note that we do not assume anything on the dependence of $\delta_\epsilon^p$ on $p\in M$. The dilations $\delta_\epsilon^p$ permit to construct the metric tangent cone of $(M,\mathcal{D},d)$ at $p\in M$. 

For $\epsilon\in(0,1]$ and $j\in\{1,\dots,k\}$, define $X_j^{p,\epsilon}:=\epsilon\cdot d\delta_{\frac1\epsilon}\circ X_j\circ\delta_\epsilon\in\Gamma(TU^p_1)$.
Then there are $X_j^{p,0}\in\Gamma(TU^p_1)$ such that $X_j^{p,\epsilon}\to X_j^{p,0}$ uniformly on compact sets as $\epsilon\to 0$.
Up to shrinking the set $U^p_1$, we can assume the convergence to be uniform on $U^p_1$.
Notice that, for all $\epsilon\in(0,1]$, $(\delta_\epsilon^p)_*X_j^{p,0} = \dd\delta_\epsilon^p\circ X_j^{p,0}\circ\delta_\epsilon^p = \epsilon^{-1} X_j^{p,0}$.

For all $\epsilon\in[0,1]$, the vector fields $X_j^{p,\epsilon}$ define a subRiemannian distance $d_\epsilon^p$ on $U^p_1$. For $\epsilon\neq 0$, the metric space $(U^p_1,d_\epsilon^p)$ is isometric to a neighborhood of $p$ in $(M,\epsilon^{-1}d)$.
As $\epsilon\to 0$, $d_\epsilon^p$ converge uniformly on $U^p_1\times U^p_1$ to $d_0^p$.
This implies that $(U^p_1,d_0^p)$ is isometric to a neighborhood of the origin in the tangent cone of $(M,d)$ at $p$. We write $d^p_0$ or just $d^p$ for the tangent distance at the point $p$.

More can be said about the tangent cone.
Let $\g_p\subset\Gamma(TU^p_1)$ be the Lie algebra generated by the vector fields $\{X_j^{p,0}\}_{j=1}^k$.
This is a finite dimensional, nilpotent, stratified Lie algebra, whose first layer is the span of the vector fields $X_1^{p,0},\dots,X_k^{p,0}$.

Recall that a Lie algebra $\g$ is \emph{stratified} or \emph{graded} of step $m$ and rank $k$ if $\g$ is a direct sum  $\g = \bigoplus_{i=1}^m V_i$ of vector subspaces $\{V_i\}_{i=1}^{m}$, with $\dim(V_1)=k$ and $[V_1,V_i]=V_{i+1}$ for all $i=1,\ldots,m$.
When we speak of a stratified Lie algebra $\g$ we mean that the stratification $V_1,\dots,V_m$  is chosen.
If $\g'=\bigoplus_{i=1}^{m'}V'_i$ is another stratified Lie algebra, then a map $A\colon\g\to\g'$ is a \emph{homomorphism of stratified (or graded) Lie algebras} if it commutes with Lie brackets and $A(V_i)\subset V'_i$.

Since a stratified Lie algebra is nilpotent, the Baker-Campbell-Hausdorff formula is a finite sum and it defines a map $*\colon \g\times\g\to\g$ that makes $(\g,*)$ into a Lie group.
More precisely, $(\g,*)$ is the unique simply connected Lie group whose Lie algebra is $\g$.
With this identification between Lie algebra and Lie group, any Lie algebra morphism is a Lie group morphism as well.

The group $\G=(\g,*)$ becomes a \emph{Carnot group} if $V_1$ is endowed by a scalar product and $\G$ is endowed with the induced left-invariant subRiemannian metric. In the case of $\g_p\subset\Gamma(TU^p_1)$, where $\g_p$ is defined above, the first layer is $V_1=\Span\{X_1^{p,0},\dots,X_k^{p,0}\}$, and the scalar product on $V_1$ is chosen by declaring that $X_1^{p,0},\dots,X_k^{p,0}$ is an orthonormal basis.

The exponential map for vector fields $\exp\colon \Gamma(TU^p_1)\to U^p_1$ (which is not globally defined), restricted to $\g_p$ gives an isometry between an open neighborhood of $0\in\G_p=\g_p$ onto an open neighborhood of $p$ in $(U^p_1,d^p_0)$.
Moreover, by results in~\cite{cr15}, we may start with some special privileged coordinates on $U^p_1$ so that they correspond to exponential coordinates of the group $\G_p$.
Therefore, if $\bar p$ is another point on another subRiemannian manifolds, and $A\colon \g_p\to\g_{\bar p}$ is a Lie algebra homomorphism, then we can see $A$ as a map $U^p_1\to U^{\bar p}_1$ that is linear in these coordinates.

\section{Calculus on subRiemannian manifolds}\label{sec:Calculus on subRiemannian manifolds}

\subsection{$P$-differential and Stepanov's theorem}\label{subsec:P-differential and Stepanov's theorem}
In this section, we introduce the notion of $P$-differentiability and recall a version of the Stepanov's theorem in the equiregular subRiemannian setting obtained recently in~\cite{gnw15}. This definition was first introduced for mappings between two Carnot groups by Pansu~\cite{p89} in his study of quasi-isometries of rank-one symmetric spaces, and was generalized later by Margulis and Mostow~\cite{mm95} for mappings between two equiregular subRiemannian manifolds (with the same dimension), where quasiconformal mappings were shown to be $P$-differentiable almost everywhere (with respect to the Popp measure). We refer the readers to~\cite[Section 4]{gnw15} and~\cite[Section 5]{flp14} for a detailed description of the $P$-differential in the equiregular subRiemannian setting. 

Let $(M,g)$ and $(N,h)$ be two equiregular subRiemannian manifolds of homogeneous dimension $n$ and rank $k$. Let $f\colon M\to N$ be a Borel mapping\footnote{Recall that a mapping $f\colon X\to Y$ is Borel if the preimage of each open set in $Y$ is a Borel set in $X$.}, $p\in M$ and $\bar p:=f(p)\in N$.

Following the notation in Section~\ref{subsec:tangent cone}, we introduce the concept of $P$-differential.
\begin{definition}[$P$-differential]\label{def1738}
	We say that $f$ is \textit{$P$-differentiable} at $p\in M$ if there exists a homomorphism of graded Lie algebras $A\colon \g_p\to\g_{\bar p}$ such that 
	\[
	\lim_{\g_p\ni X\to 0} \frac{d_h\left(\exp(A[X])(\bar p),f(\exp(X)(p))\right)}{\|X\|}=0,
	\]
	where $\|\cdot\|$ is any homogeneous semi-norm on $\g_p$. When $f$ is $P$-differentiable at $p$, we write $Df(p)$ instead of $A$ for the $P$-differential.
\end{definition}

We caution that the definition of $P$-differential depends on a choice of two systems of privileged coordinates, one centered at $p$ and the other at $\bar p$.
However, different choices of privileged coordinates ``commutes by isomorphisms". That is, if $\g'_p$ and $\bar\g'_{\bar p}$ are the graded Lie algebras that arise from a different choice of privileged coordinates, then there are isomorphisms of graded Lie algebras $F\colon \g_p\to\g'_p$ and $\bar F\colon \bar\g_{\bar p}\to\bar\g'_{\bar p}$ with the following property: 
for any map $f\colon M\to N$ with $f(p)=\bar p$ and with a $P$-differential $A\colon \g_p\to\g_{\bar p}$, a homomorphism of graded Lie algebras $A'\colon \bar\g'_{ p}\to\bar\g'_{\bar p}$ is the $P$-differential of $f$ at $p$ if and only if the following diagram commutes:
\[
\xymatrix{
	\bar\g_{ p}\ar[d]_{F}\ar[r]^A & \bar\g_{\bar p}\ar[d]^{\bar F} \\
	\bar\g'_{ p}\ar[r]_{A'} & \bar\g'_{\bar p}
}
\]

When consider a local property of our mapping $f\colon M\to N$, it is convenient to identify a local neighborhood of a point in an equiregular subRiemannian manifold with a ball in $\R^n$ in the following way: for each point $p\in M$, fix a neighborhood $U_p$ of $p$ and a smooth diffeomorphism $\varphi\colon U_p\to B(0,r)\subset \R^n$. Then $\varphi$ induces a natural subRiemannian structure on $\R^n$ by pushing forward the corresponding subRiemannian structure on $M$ so that $\varphi$ becomes an isometry with respect to the subRiemannian distance on $M$ and the induced subRiemannian distance on $\R^n$. With a slight abuse of notation, in privileged coordinates, the $P$-differential is a linear mapping $Df(0)=Df(p)\colon \R^n\to\R^n$ such that 
\begin{equation}\label{eq:P-diff in coordinate}
\lim_{y\to 0} \frac{d_h(Df(0)y, f(y))}{\|y\|} = 0
\end{equation}
where $\|\cdot\|$ is any homogeneous semi-norm on $\R^n$.
By the Ball-Box theorem (see e.g.~\cite{Gromov96}), it follows that
\begin{equation*}
	\lim_{y\to 0} \frac{d_h(Df(0)y, f(y))}{d_g(0,y)} = 0,
\end{equation*}
or, in other words,
\begin{equation}\label{eq:P-diff in coordinate 2}
 d_h(f(p),f(y)) = d_h(f(p),Df(0)y) + o\left(d_g(0,y) \right).
\end{equation}
Indeed, privileged coordinates identify a neighborhood of a point $p$ with a neighborhood of the origin $0$ in the Carnot group tangent to $M$ at $p$ endowed with the exponential coordinates.

The following Stepanov-type result was obtained in~\cite{gnw15}.
\begin{theorem}[Theorem A,~\cite{gnw15}]\label{thm:Differentiability Stepanov}
Let $f\colon (M,g)\to (N,h)$ be a Borel mapping between two equiregular subRiemannian manifolds.
	Then $f$ is $P$-differentiable at $\Vol_M$-a.e. 
	\[
	L(f) := \left\{x\in M: \limsup_{y\to x} \frac{d_h(f(y),f(x))}{d_g(y,x)} < \infty \right\}.
	\]
\end{theorem}
In reference~\cite{gnw15} this notion of differentiability is referred as geometric differentiability. This theorem is a subRiemannian counterpart of Lemma~\ref{lemma:differentiability}, which we will prove independently later.

For our later purpose of this paper, it is important to introduce the notion of a contact mapping.
\begin{definition}[Contact mappings]\label{def:contact map}
	A mapping $f\colon M\to N$ between two equiregular subRiemannian manifolds is said to be \textit{contact} if $f$ maps the horizontal distribution to horizontal distribution almost everywhere, i.e.,
	$$
	f_*\colon \mD_p\to \mD_{f(p)}\quad \Vol_M\text{-a.e. } p\in M.
	$$
\end{definition}

Note that if $f$ is $P$-differentiable at $p\in M$, then $f_*\colon \mathcal{D}_p\to \mathcal{D}_{f(p)}$; see~\cite[Section 4]{gnw15} or~\cite[Section 5.5]{flp14}. It follows immediately that $f\colon M\to N$ is contact whenever it is $P$-differentiable $\Vol_M$-a.e. in $M$.

\subsection{Sobolev spaces}\label{subsec:Sobolev spaces}
We introduce next the Sobolev spaces of mappings between two subRiemannian manifolds based on the \textit{upper gradients}. More detailed descriptions of Sobolev spaces based on this approach can be found in~\cite{s00,hkst12}.

Let $\Gamma$ be a family of curves in an equiregular subRiemannian manifold $M=(M,g)$. A Borel function $\rho\colon M\rightarrow [0,\infty]$ is \textit{admissible} for $\Gamma$ if for every locally rectifiable curve $\gamma\in \Gamma$,
\begin{equation}\label{admissibility}
\int_\gamma \rho\,ds\geq 1\text{.}
\end{equation}
The \textit{$p$-modulus} of $\Gamma$, $p\geq 1$, is defined as
\begin{equation*}
\modulus_p(\Gamma) = \inf \left\{ \int_M \rho^p\,d\Vol_M:\text{$\rho$ is admissible for $\Gamma$} \right\}.
\end{equation*}

A family of curves is called \textit{$p$-exceptional} if it has $p$-modulus zero. We say that
a property of curves holds for \textit{$p$-a.e. curve} if the collection of curves for which
the property fails to hold is $p$-exceptional.

\begin{definition}\label{def:upper gradient}
	A Borel function $g\colon M\rightarrow [0,\infty]$ is called an \textit{upper gradient} for a map $f\colon M\to N$ if for every rectifiable curve $\gamma\colon [0,1]\to M$, we have the inequality 
	\begin{equation}\label{ugdefeq}
	\int_\gamma g\,ds\geq d_h\left(f(\gamma(1)),f(\gamma(0))\right)\text{.}
	\end{equation}
	Here $ds$ is the length element on $M$. If inequality \eqref{ugdefeq} merely holds for $p$-almost every curve, then $g$ is called a \textit{$p$-weak upper gradient} for $f$.  When the exponent $p$ is clear, we omit it. 
\end{definition}

The concept of upper gradient were introduced in~\cite{hk98}. It was initially called ``very weak gradient", but the befitting term ``upper gradient" was soon suggested. Functions with $p$-integrable $p$-weak upper gradients were subsequently studied in~\cite{km98}. By \cite[Lemma 6.2.2]{hkst12}, a mapping $f$ has a $p$-weak upper gradient in $L^p_{loc}(M)$ if and only if it has an actual upper gradient in $L^p_{loc}(M)$. 

A $p$-weak upper gradient $g$ of $f$ is \textit{minimal} if for every $p$-weak upper gradient $\tilde{g}$ of $f$ there holds that $\tilde{g}\geq g$ $\Vol_M$-almost everywhere. If $f$ has an upper gradient in $L^p_{loc}(M)$, then $f$ has a unique (up to sets of $\Vol_M$-measure zero) minimal $p$-weak upper gradient by~\cite[Theorem 6.3.20]{hkst12}.  In this situation, we denote the minimal upper gradient by $g_{f}$.

In view of the above result, the minimal $p$-weak upper gradient $g_u$ of a function $u$ on $M$ should be thought of as a substitute for $|\nabla_H u|$, or the length of a horizontal gradient, for functions defined on subRiemannian manifolds.

Fix a Banach space $\mathbb{V}$, and we first define the Sobolev space $N^{1,p}(M,\mathbb{V})$ of $\mathbb{V}$-valued mappings. Let $\tilde{N}^{1,p}(M,\mathbb{V})$ denote the collection of all maps $f\in L^p(M,\mathbb{V})$ that have an upper gradient in $L^p(M)$. We equip it with seminorm
\begin{equation*}
\|f\|_{\tilde{N}^{1,p}(M,\mathbb{V})}=\|f\|_{L^p(M,\mathbb{V})}+\|g_f\|_{L^p(M)},
\end{equation*}
where $g_f$ is the minimal $p$-weak upper gradient of $f$. We obtain a normed space $N^{1,p}(M,\mathbb{V})$ by passing to equivalence classes of functions in $\tilde{N}^{1,p}(M,\mathbb{V})$ with respect to equivalence relation: $f_1\sim f_2$ if $\|f_1-f_2\|_{\tilde{N}^{1,p}(M,\mathbb{V})}=0$. Thus
\begin{equation}\label{eq:definition of Nowton-Sobolev space}
N^{1,p}(M,\mathbb{V}):=\tilde{N}^{1,p}(M,\mathbb{V})/\{f\in \tilde{N}^{1,p}(M,\mathbb{V}): \|f\|_{\tilde{N}^{1,p}(M,\mathbb{V})}=0\}.
\end{equation}

Let $\tilde{N}_{loc}^{1,p}(M,\mathbb{V})$ be the vector space of functions $f\colon M\to \mathbb{V}$ with the property that every point $x\in M$ has a neighborhood $U_x$ in $M$ such that $f\in \tilde{N}^{1,p}(U_x,\mathbb{V})$. Two functions $f_1$ and $f_2$ in $\tilde{N}_{loc}^{1,p}(M,\mathbb{V})$ are said to be equivalent if every point $x\in M$ has a neighborhood $U_x$ in $M$ such that the restrictions $f_1|_{U_x}$ and $f_2|_{U_x}$ determine the same element in $\tilde{N}^{1,p}(U_x,\mathbb{V})$. The local Sobolev space $N_{loc}^{1,p}(M,\mathbb{V})$ is the vector space of equivalent classes of functions in $\tilde{N}_{loc}^{1,p}(M,\mathbb{V})$ under the preceding equivalence relation.

To define the Sobolev space $N^{1,p}(M,N)$ of mappings $f\colon M\to N$, we first fix an isometric embedding $\varphi$ of $N$ into some Banach space $\mathbb{V}$. Then the Sobolev space $N^{1,p}(M,N)$ consists of all mappings $f\colon M\to N$ with $\varphi\circ f\in N^{1,p}(M,\mathbb{V})$.

For every subRiemannian manifold $(M,g)$, let $\mathcal{C}_\varepsilon(T)$ be the collection of curves $\gamma$ such that $d_g(\gamma(b),\gamma(a))\geq \varepsilon$. We need the following useful characterization of Sobolev spaces.
\begin{theorem}[Theorem 3.10, \cite{w12proc}]\label{thm:characterization of Sobolev spaces}
	Let $p>1$ and let $M$ be compact. Then $f\in N^{1,p}(M,N)$ if and only if
	\begin{align}\label{eq:mod for Sob regularity}
	\liminf_{\varepsilon\to 0}\varepsilon^{p}\Modd_p\big(f^{-1}(\mathcal{C}_\varepsilon(N))\big)<\infty.
	\end{align}
	Moreover, if this is the case, then the liminf on the left-hand side is an actual limit, and we have
	\begin{align*}
	\|g_f\|_{L^p(M)}^p=\lim_{\varepsilon\to 0} \varepsilon^{p}\Modd_p\big(f^{-1}(\mathcal{C}_\varepsilon(N))\big)
	\end{align*}
\end{theorem}

The following differentiability result will be used frequently in our later sections.
\begin{proposition}[Proposition 5.16,~\cite{gnw15}]\label{prop:on differentiability}
	Let $f\colon M\to N$ be an open mapping between two equiregular subRiemannian manifolds of homogeneous dimension $Q\geq 2$. If $f\in N^{1,Q}_{loc}(M,N)$, then $f$ is $P$-differentiable a.e. in $M$.
\end{proposition}

We also remark that if a continuous open mapping $f\colon M\to N$ is $P$-differentiable at $p\in M$ and $g\colon f(M)\to N'$ is $P$-differentiable at $f(p)\in N$, then it follows that $g\circ f$ is $P$-differentiable at $p\in M$ and $D(g\circ f)(p)=Dg(f(p))\circ Df(p)$, cf.~\cite[Lemma 5.8]{flp14}.

\section{Quasiregular mappings}\label{sec:Quasiregular mappings}
In this section, we will present various natural definitions of quasiregular mappings $f \colon (M,g)\to (N,h)$, and discuss their basic analytic properties. We will present here four different definitions in total, and the remaining two definitions appearing in our main theorem, Theorem~\ref{thm:equivalence of quasiregular mappings}, will be introduced in Section~\ref{sec:The Popp extension of a horizontal metric}.

Throughout this section, $M$ and $N$ are equiregular subRiemannian $n$-manifolds of homogeneous dimension $Q\geq 2$ with horizontal distributions $\mathcal{D}_M$ and $\mathcal{D}_N$ equipped with horizontal metrics $g$ and $h$ respectively. We equip the manifolds with Popp measures $\Vol_M$ and $\Vol_N$, and the meaning of a.e. is understood with respect to the corresponding Popp measure. 

\subsection{Definitions of quasiregular mappings}\label{subsec:Definitions of quasiregular mappings}

We first recall the so-called \emph{weak} metrically quasiregular mappings. For this, we set $h_f(x)=\liminf_{r\to 0}H_f(x,r)$, where $H_f(x,r)$ is defined in the introduction in~\eqref{pw_distort}. The metrically quasiregular mappings were introduced in Definition~\ref{def:metric quasiregular map}.
\begin{definition}[Weak metrically quasiregular mappings]\label{def:weak metric qr}
A branched covering $f\colon (M,g)\to (N,h)$ is said to be \textit{weakly metrically $H$-quasiregular} if it is constant or if it satisfies the following two conditions 
\begin{itemize}
\item[1)] $h_f(p)<\infty$ for all $p\in M$,
\item[2)] $h_f(p)\leq H$ for almost every $p\in M$. 
\end{itemize}
We say that $f\colon (M,g)\to (N,h)$ is \textit{weakly metrically quasiregular} if it is weakly metrically $H$-regular for some $1\leq H<\infty$. 
\end{definition}

If $f\colon M\to N$ is a continuous mapping, $q\in N$ and $A\subset M$, we use the notation
\begin{equation}\label{def:multiplicity function}
N(q,f,A)=\text{card}\{f^{-1}(q)\cap A\}
\end{equation}
for the multiplicity function.

\begin{definition}[Analytically quasiregular mappings]\label{def:analytic def for qr}
A branched covering $f\colon (M,g)\to (N,h)$ is said to be \textit{analytically $K$-quasiregular} if $f\in N^{1,Q}_{loc}(M,N)$ and 
\begin{align*}
g_f^Q(p)\leq KJ_f(p)
\end{align*}
for $\Vol_M$-a.e. $p\in M$.
\end{definition}

The \emph{volume Jacobian} above is given by
\begin{equation}\label{eq:analytic jacobian}
J_f(p)=\frac{d\big(f^*\Vol_N\big)(p)}{d\Vol_M(p)},
\end{equation}
where the \textit{pullback} $f^*\Vol_N$ is defined as
\begin{equation}\label{eq:definition of pull-back}
f^*\Vol_N(A)=\int_N N(q,f,A)d\Vol_N(q).
\end{equation}

We point out an useful observation by Williams~\cite{w12} that the Jacobian $J_f(p)$ from~\eqref{eq:analytic jacobian} can be alternately described by
\begin{equation}\label{eq:another description of Jacobian}
J_f(p)=\lim_{r\to 0}\frac{\Vol_N\big(f(B(p,r))\big)}{\Vol_M(B(p,r))}
\end{equation}
for a.e. $p\in M$; see also~\cite[Section 4.3]{g14} for a simple proof of this fact.

\begin{definition}[Geometrically quasiregular mappings]\label{def:geometric def for qr} 
A branched covering $f\colon (M,g)\to (N,h)$ is said to be \textit{geometrically $K$-quasiregular} if it satisfies the so-called \textit{$K_O$-inequality}: For each open set $\Omega\subset M$ and for each curve family $\Gamma$ in $\Omega\subset M$ and for any admissible function $\rho$ for $f(\Gamma)$, there holds that
\begin{align*}
\Modd_Q(\Gamma)\leq K\int_{N}N(q,f,\Omega)\rho^Q(q)d\Vol_N(q). 
\end{align*}
\end{definition}

\subsection{Some basic facts about quasiregular mappings in subRiemannian manifolds}\label{subsec:Some basic facts about quasiregular mappings in subRiemannian manifolds}

We proceed by collecting some results from the paper~\cite[Section 5]{gnw15} about quasiregular mappings between equiregular subRiemannian manifolds.

\begin{proposition}[Proposition 5.15,~\cite{gnw15}]\label{prop:Jacobian=det Df}
	Let $f\colon (M,g)\to (N,h)$ be a branched covering such that $f$ is $P$-differentiable at $p_0\in M$, then 
	\begin{align*}
	J_{Df(p_0)}(0)=J_f(p_0).
	\end{align*}
Here $Df(p_0)$ is the $P$-differential of $f$ at $p_0$.
\end{proposition}

Let $p_0\in M$. If a mapping $f\colon (M,g)\to (N,h)$ is $P$-differentiable at $p_0$, then the $P$-differential $Df(p_0)$ maps the horizontal distribution $\mD_{p_0}$ on $M$ to the horizontal distribution $\mD_{f(p_0)}$ on $N$. This fact allows to define the maximal norm of the differential $Df(p_0)$ as 
\begin{align*}
\|Df(p_0)\|&=\max_v\Big\{|Df(p_0)v|_h:v\in \mathcal{D}_{p_0}\text{ and } |v|_g=1 \Big\}\numberthis\label{eq:def maximal norm}\\
&=\limsup_{r\to 0}\frac{L_f(p_0,r)}{r}
\end{align*}
Similarly, the minimal norm of the differential $Df(p_0)$ is defined as
\begin{align*}
\|Df(p_0)\|_s&=\min_{v}\Big\{|Df(p_0)v|_h:v\in \mathcal{D}_{p_0} \text{ and } |v|_g=1 \Big\}\numberthis\label{eq:def for minimal norm}\\
&=\liminf_{r\to 0}\frac{l_f(p_0,r)}{r}.
\end{align*}

\begin{lemma}[Lemma 5.18,~\cite{gnw15}]\label{lemma:simple}
	If $f\colon M\to N$ is weak metrically $H$-quasiregular, then 
	\begin{align}\label{eq:bound on norm in tangent cone}
	\frac{\|Df(p_0)\|}{\|Df(p_0)\|_s}\leq H
	\end{align} 
	for a.e. $p_0\in M$.
\end{lemma}

Let $f\colon M\to N$ be a continuous mapping. The pointwise Lipschitz constant of $f$ at $p_0\in M$ is defined as
\begin{align*}
\Lip f(p_0)=\limsup_{p\to p_0}\frac{d_h(f(p),f(p_0))}{d_g(p,p_0)}. 
\end{align*}\

\begin{lemma}[Lemma 5.19,~\cite{gnw15}]\label{lemma:Lip=norm of differential}
	Let $f\colon (M,g)\to (N,h)$ be a mapping $P$-differentiable at $p_0\in M$. Then 
	\begin{align}\label{Lip=norm of differential}
	\Lip f(p_0)=\|Df(p_0)\|.
	\end{align} 
\end{lemma}

\begin{proposition}\label{prop:minimal upper gradient}
	Let $f\colon (M,g)\to (N,h)$ be an $N^{1,Q}_{loc}(M,N)$-mapping that is $P$-differentiable a.e. in $M$. Then $\Lip f$ is the minimal $Q$-weak upper gradient of $f$, \ie
	\begin{align}\label{eq:minimal upper gradient}
	g_f=\Lip f.
	\end{align}
\end{proposition}
\begin{proof}
	The result is essentially proved in~\cite[Proposition 5.20]{gnw15}. Indeed, the assumption that $f$ belongs to $N^{1,Q}_{loc}(M,N)$ implies that $\Lip f$ is an $Q$-weak upper gradient of $f$ (cf.~\cite[Lemma 5.1]{w12proc}). On the other hand, the assumption that $f$ is $P$-differentiable a.e. implies, by the proof of Proposition 5.20 in~\cite{gnw15}, that $g_f\geq \Lip f$.
\end{proof}

\begin{lemma}[Proposition 4.18,~\cite{g14}]\label{lemma:positiveness of Jacobian}
	If $f\colon (M,g)\to (N,h)$ is analytically quasiregular, then $J_f(p)>0$ for a.e. $p\in M$.
\end{lemma}

The following estimate on the volume Jacobian is well-known in different setting; see for instance~\cite{w12proc,w14,w12}. 
\begin{lemma}\label{lemma:bounds on Jacobian}
	Let $f\colon (M,g)\to (N,h)$ be a branched covering that is $P$-differentiable a.e. in $M$. Then for each $p\in M$, there exist a radius $r_p>0$ and a constant $C$, depending only on $p$ and $f(p)$, such that for $\Vol_M$-a.e. $p_0\in B(p,r_p)$
	\begin{align*}
	C^{-1}\|Df(p_0)\|_s^Q\leq J_f(p_0)\leq C\|Df(p_0)\|^Q.
	\end{align*}
\end{lemma}

\begin{proof}
The claim is a consequence of the fact that equiregular manifolds of homogeneous dimension $Q$ are locally Ahlfors $Q$-regular. Indeed, for each point $p$, we may find $r_p>0$ and $r_{f(p)}>0$ such that the balls $B(p,r_p)$ and $B(f(p), r_{f(p)})$ are relatively compact and both the metric measure spaces $\big(B(p,r_p),d_g,\Vol_M\big)$ and $\big(B(f(p),r_{f(p)}),d_h,\Vol_N\big)$ are Ahlfors $Q$-regular with constant $C$, i.e., 
\begin{align*}
	C^{-1}r^Q\leq \Vol_M(B(p_0,r))\leq Cr^Q
\end{align*}	
for each $p_0\in B(p,r_p)$ with $B(p_0,r)\subset B(p,r_p)$ and 
\begin{align*}
C^{-1}r^Q\leq \Vol_N\big(B(f(p_0),r)\big)\leq Cr^Q
\end{align*}
for each $B(f(p_0),r)\subset B(f(p),r_{f(p)})$.

Fix such a point $p_0$ with the additional property that~\eqref{eq:another description of Jacobian} holds and fix an $\varepsilon>0$. Then it follows that there exists $r_0>0$ such that for each $r\in (0,r_0)$
	\begin{align*}
	f\big(B(p_0,r)\big)\subset B\big(f(p_0),r(\Lip f(p_0)+\varepsilon)\big).
	\end{align*}
	Thus
	\begin{align}\label{eq:bj 1}
	\frac{\Vol_N\big(f(B(p_0,r))\big)}{\Vol_M\big(B(p_0,r)\big)}\leq \frac{\Vol_N\big(B\big(f(p_0),r\,(\Lip f(p_0)+\varepsilon)\big)\big)}{\Vol_M\big(B(p_0,r)\big)}.
	\end{align}
When $r$ is sufficiently small, we have
$$\Vol_N\big(B\big(f(p_0),r\,(\Lip f(p_0)+\varepsilon)\big)\big)\leq Cr^Q(\Lip f(p_0)+\varepsilon)^Q$$\
and
$$\Vol_M\big(B(p_0,r)\big)\geq C^{-1}r^Q.$$
Consequently, we conclude that
\begin{align*}
	J_f(p_0)&=\lim_{r\to 0}\frac{\Vol_N\big(f(B(p_0,r))\big)}{\Vol_M\big(B(p_0,r)\big)}\\
	&\leq \lim_{r\to 0}\frac{Cr^Q(\Lip f(p_0)+\varepsilon)^Q}{C^{-1}r^Q}\leq C^2(\Lip f(p_0)+\varepsilon)^Q.
\end{align*}
The right-hand side inequality of the claim follows immediately by sending $\varepsilon\to 0$. The proof of the left-hand side is similar and we omit it here.
\end{proof}

\begin{remark}\label{rmk:section 4}

i) We are not aware if there is a simple argument that shows the constant $C$ appearing in Lemma~\ref{lemma:bounds on Jacobian} can be taken to be $1$. This improvement would be useful when one tries to show the equivalences of 1), 2), 4) and 5) in Corollary~\ref{coro:equivalence of 1-quasiregular mappings} directly, since the inequalities become equalities for 1-quasiregular mappings. On the other hand, these estimates are not sharp for general $K$-quasiregular mappings. To obtain sharper relations of the estimates in Lemma~\ref{lemma:bounds on Jacobian}, we will use the local Popp extensions introduced in the next section.

ii) Note that if $f\in N^{1,Q}_{loc}(M,N)$ is a branched covering with $J_f>0$ a.e. in $M$, then it follows from Proposition~\ref{prop:on differentiability} and Proposition~\ref{prop:Jacobian=det Df} that $f$ is $P$-differentiable at a.e. $p\in M$ and the $P$-differential $Df(p):\mathbb{G}_p\to \mathbb{G}_{f(p)}$ is a Carnot group isomorphism. In particular, $f$ is a contact mapping and the pushforward $f_*:\mD_M\to\mD_N$ is injective a.e. in $M$. Recall also that (cf.~Section~\ref{subsec:Sobolev spaces}) the standard chain rule holds for our mapping $f$. We will use these facts in Section~\ref{subsec:compare two popp extensions}.

\end{remark}

\section{The Popp extension of a horizontal metric}\label{sec:The Popp extension of a horizontal metric}

\subsection{The construction of a Popp extension}
In this section we construct local extensions for a horizontal metric on an equiregular subRiemannian manifold. An extension of a horizontal metric is a local Riemannian metric on the manifold. Any such extension will be called a \textit{Popp extension} of the horizontal metric or a \textit{local Popp extension} if we wish to emphasize the localness of the construction. A Popp extension is constructed from the Lie bracket structure of a subRiemannian manifold together with its horizontal metric. 

The construction depends on an adapted frame, and in general, the construction cannot be applied consistently globally on a general subRiemannian manifold~\footnote{We thank Prof.~Enrico Le Donne for pointing out this to us.}. We will work out the reason for this in some detail.

Even though Popp extensions themselves are not globally well-defined, we can still use them to define distortion between two horizontal metrics invariantly and globally on an equiregular subRiemannian manifold. The distortion is encoded in a matrix valued function on the subRiemannian manifold that is constructed from two Popp extensions, and which has a well defined eigenvalue equation. Estimating the  corresponding eigenvalues turns out to be the key ingredient in our main proof of the equivalences of various definitions of quasiregular mappings between subRiemannian manifolds. 

We begin by constructing a local Popp extension for a given horizontal metric and by finding a general formula for the Popp extension. Given a point on the manifold, a Popp extension is a local inner product for tangent vectors of the manifold near the point. Thus we can consider it as a local Riemannian metric. 

A Popp extension is the inner product constructed alongside with the construction of the Popp measure. The construction is well known, but we find explicit formulas for the extension are hard to come by in the literature. We will closely follow the extremely clear presentation of Barilari and Rizzi~\cite{br13} on this matter.

Our construction of the Popp extension yields equivalent formulas to those derived in~\cite{br13}. The main difference is that we construct the Popp extension with respect to an adapted frame $X=(X_1,\ldots,X_n)$ \textit{without the assumption} that the horizontal part of the frame $(X_1,\ldots,X_k)$  is orthonormal. (Here $n=\dim M$ and $k$ is the rank of the distribution $\mD$.) The reason for lifting the assumption of orthonormality of $(X_1,\ldots,X_k)$ is that the distortion function, which we will define later, depends on two different horizontal metrics. Of course, in this case the frame $(X_1,\ldots,X_k)$ for $\mD$ is in general non-orthonormal with respect to at least one of the horizontal metrics.

We now recall a few definitions. As introduced in Section~\ref{sec:Geometry of equiregular subRiemannian manifolds}, a subRiemannian manifold admits a nilpotentization, which is the graded vector space
$$
gr(\mD)=\mD\oplus\mD^2/\mD\oplus\cdots\oplus\mD^m/\mD^{m-1}.
$$
The whole tangent bundle of the subRiemannian manifold can then locally be identified with this grading by choosing an adapted frame $X=(X_1,X_2,\ldots,X_n)$ on an open set of $TM$. This identification is sometimes stated as there being a non-canonical isomorphism between $T_pM$ and $gr_p(\mD)$, $p\in M$ (see e.g.~\cite{br13}). We can rephrase this statement by saying that choosing local adapted frames to cover the manifold does not (in general) yield a global isomorphism between $TM$ and $gr(\mD)$. 

If $X$ is an adapted frame, then the equivalence class of $(X_{k_{s-1}+1},X_{k_{s-1}+2},\ldots,X_{k_{s}})$ is a basis for $\mD^s/\mD^{s-1}$, $s=1,\ldots,m$. We use the notation $\mD=\mD^1$ (also $k=k_1$) and set $\mD^{0}=0$ here and in what follows. We call $\mD^s$ the $s^{\tiny{th}}$ layer.

We use the notion of \emph{adapted structure constants} introduced in~\cite{br13}. Let us take  $X=(X_1,X_2,\ldots,X_n)$ to be an adapted basis. Then the $s^{\text{\tiny th}}$ adapted structure constants are defined by the formula
$$
[X_{i_1},[X_{i_2},\ldots,[X_{i_{s-1}},X_{i_s}]]]=b^\beta_{i_1i_2\cdots i_s}X_\beta \mbox{ mod } \mD^{s-1},
$$
where $1\leq i_1,\ldots,i_s\leq k$ and Einstein summation over $\beta=k_{s-1}+1,\ldots, k_s$ is implied. If $(\omega^1,\omega^2,\ldots,\omega^n)$ is a dual basis to $X$, then the above is equivalent to
$$
b^\a_{i_1i_2\cdots i_s}=\omega^\a\left([X_{i_1},[X_{i_2},\ldots,[X_{i_{s-1}},X_{i_s}]]]\right), \quad \a\in\{k_{s-1}+1,\cdots,k_{s}\}.
$$
This is well defined by the elementary fact that $\omega^\a$, $\a=k_{s-1}+1,\ldots, k_s$, annihilates $\mD^{s-1}$ and all the lower layers. 

\begin{proposition}(Local Popp extension)
	Let $(M,g)$ be an equiregular subRiemannian manifold of step $m$. Let $p\in M$ and let $(X_1,X_2,\ldots,X_n)$ be an adapted frame on an open neighborhood $U$ of $p$. Then in this adapted frame the Popp extension $\ol{g}$ of the horizontal metric $g$ has the block diagonal form
	$$
	\ol{g}=\left[
	\begin{array}{cccc}
	g & 0 & 0 & 0\\
	0 & g_2 & 0 &0\\
	0 & 0 & \ddots & 0 \\
	0 & 0 & 0 & g_m \\
	\end{array}\right].
	$$
	Here the block matrices $g_s$ constitute an inner product for the vector spaces spanned by $X_{k_{s-1}+1},\ldots, X_{k_{s}}$, $s=1,\ldots,m$. The inverse metric $\ol{g}^{-1}$ has blocks $g_s^{-1}$ with components
	$$
	g_s^{\a\b}=g(b^\a,b^\b), \quad \a,\b\in \{k_{s-1}+1,\ldots, k_{s}\}.
	$$
	This notation means
	$$
	g_s^{\a\b}=b^\a_{i_1i_2\cdots i_s}g^{i_1j_1}g^{i_2j_2}\cdots g^{i_sj_s}b^\b_{j_1j_2\cdots j_s},
	$$
	where we have denoted the components of the inverse of the horizontal metric by upper indices.
\end{proposition}
\begin{proof}
	This is essentially proven in~\cite{br13} (see also~\cite{m02}) in a setting where one uses a local orthonormal adapted frame for $\mD$ instead of a general local adapted frame $(X_1,\ldots,X_k)$ for $\mD$ we use. Thus, we only sketch the proof.
	
	Using an \emph{orthonormal} basis for $\mD$ it was shown in~\cite[Eq. 24]{br13} that in their notation
	$$
	B_s^{\a\b}=\sum_{i_1,i_2,\ldots,i_s}^kb^\a_{i_1i_2\cdots i_s}b^\b_{i_1i_2\cdots i_s},
	$$
	$\a,\b\in \{k_{s-1}+1,\ldots, k_{s}\}$. Going through the same calculations as there to derive this formula, we find that in general the inverse $g_s^{-1}$ of the $s^{\text{\tiny th}}$ block matrix $g_s$ has the components
	$$
	g(b^\a,b^\b). 
	$$
	Here, for $\a,\b\in \{k_{s-1}+1,\ldots, k_{s}\}$ each $b^\a$ and $b^\b$ are the $s^{\text{\tiny th}}$ adapted structure constants considered as $s$-tensors on the horizontal space $\mD$. The inner product $g(b^\a,b^\b)$ of $b^\a$ and $b^\b$ is then just 
	$$
	g(b^\a,b^\b)=b^\a_{i_1i_2\cdots i_s}g^{i_1j_1}g^{i_2j_2}\cdots g^{i_sj_s}b^\b_{j_1j_2\cdots j_s},
	$$
	where the Einstein summation is implied in the indices $i_l,j_l\in\{1,\ldots,k\}$, $l=1,\ldots,s$.
	That is, by the usual convention of denoting by upper indices the components of the inverse of a Riemannian metric, we have that
	$$
	g_s^{\a\b}=g(b^\a,b^\b)=b^\a_{i_1i_2\cdots i_s}g^{i_1j_1}g^{i_2j_2}\cdots g^{i_1j_1}b^\b_{j_1j_2\cdots j_s}.
	$$
	We have now constructed the inverses of the blocks of the Popp extension $\ol{g}$ of $g$. The matrices $(g_s^{\a\b})$ are invertible~\cite[p.9]{br13} concluding the proof.
\end{proof}

We proceed by studying how a Popp extension transforms under a change of adapted frames.
\begin{proposition}\label{local_Popp}
	Let $(M,g)$ be a subRiemannian manifold. Let $(X_1,X_2,\ldots,X_n)$ and $(Y_1,Y_2,\ldots,Y_n)$ be two sets of adapted frames near $p\in M$. Let $T$ be the change of frame transformation from $X$-frame to $Y$-frame, i.e.,
	$$
	Y_i=T_i^{j}X_j.
	$$
	Then $T$ is lower block triangular
	$$
	T=\left[
	\begin{array}{cccc}
	T & 0 & 0 & 0\\
	\# & T_2 & 0 &0\\
	\# & \# & \ddots & 0 \\
	\# & \# & \# & T_m \\
	\end{array}\right].
	$$
	
	Let $\ol{g}$ and $\widetilde{g}$ be the Popp extensions of $g$ constructed with respect to frames $X$ and $Y$ respectively. Then the blocks $g_s$ and $\widetilde{g}_s$, $s=1,\ldots,m$, of $\ol{g}$ and $\widetilde{g}$ satisfy
	$$
	g_s=T_s^T\widetilde{g}_s T_s.
	$$
\end{proposition}

\begin{proof}[Proof of Proposition~\ref{local_Popp}]
	Since both of the frames $X$ and $Y$ are adapted, $T$ is lower block triangular with respect to basis $X$ and $Y$. We denote the diagonal blocks of $T$ by $T_s$, with the convention that $T_1=T$. 
	
	Let $s=1,\ldots,m$ and let $\ol{g}$ and $\widetilde{g}$ be the Popp extensions of $g$ relative to frames $X=(X_1,\ldots,X_n)$ and $Y=(Y_1,\ldots,Y_n)$ respectively. We make the following convention. We denote the adapted structure constants with respect to the frame $X$ by
	$$
	b^\a_{i_1\cdots i_s}
	$$
	and we use the (physicists') primed indices notation
	$$
	b^{\a'}_{i'_1\cdots i'_s}
	$$
	to denote the adapted structure constants with respect to the other frame $Y$. We also denote the components of $g_s^{-1}$ and $\widetilde{g}_s^{-1}$ by $g_s^{\a\b}$ and $g_s^{\a'\b'}$ respectively. Here $\a,\a'\b,\b'=k_{s-1}+1,\ldots, k_s$.  We also remind that we use the Einstein summation convention.
	
	We note that if $\omega=(\omega^1,\ldots,\omega^n)$ is a dual basis of $X$ and $\eta=(\eta^1,\ldots,\eta^n)$ is a dual basis of $Y$, then we have the standard formula
	$$
	\eta=T^{-T}\omega
	$$
	on how a change of basis induces a change of the dual basis. Since $T^{-T}$ is upper block triangular, it follows that
	\begin{align}\label{b_trans}
		b^{\a'}_{i'_1\cdots i'_s}&=\eta^{\a'}([X_{i'_1},[X_{i'_2},\ldots,[X_{i'_{s-1}},X_{i'_s}]]])=(T_s^{-T}\omega(s))^{\a'}[X_{i'_1},[X_{i'_2},\ldots,[X_{i'_{s-1}},X_{i'_s}]]] \nonumber \\
		&+ \mbox{ a linear combination of $\omega_l$, $l > n_s$}, \mbox{operating on } [X_{i'_1},[X_{i'_2},\ldots,[X_{i'_{s-1}},X_{i'_s}]]] \nonumber \\
		&= (T_s^{-T}\omega(s))^{\a'}([X_{i'_1},[X_{i'_2},\ldots,[X_{i'_{s-1}},X_{i'_s}]]].
	\end{align}
	Here we have denoted
	$$
	\omega(s)= (\omega^{k_{s-1}+1},\ldots,\omega^{k_s})
	$$
	and we understand that the notion
	$$
	(T_s^{-T}\omega(s))^{\a'}
	$$
	refers to the $(\a'-k_{s-1})^{\tiny th}$ element of the $(k_s-k_{s-1})$-vector $T_s^{-T}\omega(s)$. 
	
	The last equality in~\eqref{b_trans} holds since dual vectors $\omega^l$, $l > n_s$, annihilate all the lower layers $\mD_t$, $t\leq s$. This follows from the definition of a dual frame and from the fact that the frame itself is adapted. 
	Thus we have
	$$
	b^{\a'}_{i'_1\cdots i'_s}=(T_s^{-T})^{\a'}_{\ph \a}b^{\a}_{i'_1\cdots i'_s}.
	$$
	Here we have denoted the components of the blocks of $T$ by $(T_s)^{\ph \a}_{\a'}$ thus satisfying
	$$
	Y_{\a'}=(T_s)^{\ph \a}_{\a'} X_{\a}, \quad \a,\a'\in \{k_{s-1}+1,\ldots, k_{s}\}
	$$
	and
	$$
	(T_s^{-T})^{\a'}_{\ph \a}=(T^{-1}_s)_{\a}^{\ph \a'}.
	$$
	
	We now have that
	$$
	g_s^{\a\b}=g(b^\a,b^\b),
	$$
	and by the multilinearity of the nested brackets we also have
	$$
	b^{\a'}_{i'_1i'_2\cdots i'_s}=(T^{-1}_s)_{\a}^{\ph \a'} b^{\a}_{i_1 i_2\cdots i_s} \,  T^{\ph i_1}_{i'_1}T^{\ph i_2}_{i'_2}\cdots T^{\ph i_s}_{i'_s}.
	$$
	By the tensoriality of the horizontal metric we have
	$$
	g^{i'j'}=(T^{-1})^{\ph i'}_{i} (T^{-1})^{\ph j'}_{j} g^{ij}.
	$$
	Consequently, plugging in these transformation rules, we have that
	$$
	g_s^{\a'\b'}=(T^{-T}_s)_{\a}^{\ph \a'} \tilde{g}_s^{\a\b} (T^{-T}_s)_{\b}^{\ph \b'}.
	$$
	That is, 
	$$
	\widetilde{g}_s=T_s^T g_s T_s.
	$$
\end{proof}

We record the following fact: If the Popp extension $\ol{g}$ of $g$ would be a proper $2$-tensor field on $M$ it would change under the change of an adapted frame (or any frame in general) as $\ol{g}\hookrightarrow T^{T}\ol{g} T$. Proposition~\ref{local_Popp} above shows that instead $\ol{g}$ transforms as
$$
\ol{g} \hookrightarrow\mbox{diag}(T)^T\ol{g}\ \mbox{diag}(T).
$$
Here $\mbox{diag}(T)$ stands for the block diagonal matrix with the block matrices $T_i$, $i=1,\ldots,m$, as its diagonal blocks. It follows that the Popp extension does not constitute a well defined inner product for $TM$, but instead yields an inner product for $gr(\mD)$. 

\begin{corollary}
	Let $(M,g)$ be an equiregular subRiemannian manifold with horizontal metric $g$. The Popp extension $\ol{g}$ of $g$ defines an inner product for $gr(\mD)$. 
\end{corollary}
\begin{proof}
	Cover $M$ by an open sets where on each open set an adapted frame is defined. Adapted frame induces a local frame for $gr(\mD)$. The transformation rule 
	$$
	\ol{g} \hookrightarrow\mbox{diag}(T)^T\ol{g}\ \mbox{diag}(T).
	$$
	shows that on the overlap of the open sets, the inner product defined with respect to different adapted frames agree.	This completes the proof.
\end{proof}

We remark that a local Popp extension defined on an open neighborhood of $p\in M$ can be continued (from a slightly smaller open neighborhood $V$ of $p$) to a global Riemannian metric on $M$ by using a partition of unity. Of course this continuation then has little to do with the subRiemannian structure or the horizontal metric outside $V$ in this case.

As we go on, we will see that even though the Popp extension is local, it contains useful information for the study of quasiregularity in the subRiemannian setting.

\subsection{Comparing two Popp extensions}\label{subsec:compare two popp extensions}
We will later give a new definition of quasiregular mappings between subRiemannian manifolds using two Popp extensions $\ol{g}$ and $\ol{h}$ of two horizontal metrics $g$ and $h$. To show that the resulting definition is equivalent to the standard ones in the literature, we will need to compare eigenvalues of the \emph{distortion matrix} $\ol{g}^{-1}\ol{h}$ of the Popp extensions of two horizontal metrics $g$ and $h$. The notation $\ol{g}^{-1}$ here means the inverse of $\ol{g}$. 

The approach is a natural generalization of the one introduced in~\cite{l14} to define quasiregular mappings between Riemannian manifolds. In~\cite{l14} the distortion matrix is called a \emph{distortion tensor}, since in that case $\ol{g}^{-1}\ol{h}$ is a true $\binom{1}{1}$-tensor field on $M$ in the standard sense. Even though in our situation $\ol{g}^{-1}\ol{h}$ is (typically) not a tensor field, it still has well defined eigenvalues.

\begin{proposition}\label{evs_inv}
	Let $M$ be an equiregular subRiemannian manifold and let $g$ and $h$ be two horizontal metrics on it. The eigenvalues of the distortion matrix $\ol{g}^{-1}\ol{h}$ are well defined, i.e. independent of the choice of the adapted frame used in the construction of the Popp extensions $\ol{g}$ and $\ol{h}$. In particular, the eigenvalues of the distortion matrix are the same in \emph{any} frame.
\end{proposition}
\begin{proof}
	Let $\ol{g},\ol{h}$ be local Popp extensions of $g$ and $h$ constructed with respect to the (same) adapted frame $X$, and let $\widetilde{g}, \widetilde{h}$ be those constructed with respect to another adapted frame $Y$. The eigenvalues of $\ol{g}^{-1}\ol{h}$ are defined by its diagonal blocks $g_s^{-1}h_s$, $s=1,\ldots,m$.
	
	Under the change of adapted basis, we have by Proposition~\ref{local_Popp} that
	$$
	\widetilde{g}^{-1}\widetilde{h}=T^{-1}\ol{g}^{-1}\ol{h}T,
	$$
	where $T$ (and thus also $T^{-1}$) is lower block triangular. Thus the matrix on the right hand side is lower block diagonal. It follows that its eigenvalues are determined by the eigenvalues of its diagonal blocks.  
	
	These diagonal blocks are of the form
	$$
	T_s^{-1}g_s^{-1}h_sT_s.
	$$
	These have the same eigenvalues as $g_s^{-1}h_s$ by similarity. Thus the eigenvalues are well defined.
	
	The latter claim follows since \emph{any} change of frames induces a similarity transformation of $\ol{g}^{-1}\ol{h}$ which we have now shown to have a well defined eigenvalue equation. This concludes the proof.
\end{proof}

The following result, which gives estimates for the eigenvalues of the distortion matrix, will be of key importance in our later proofs. 
\begin{proposition}\label{eigenval_compr}
	Let $g$ and $h$ be two horizontal metrics. Let us denote the eigenvalues of $g^{-1}h$ by $\lambda_1\leq\cdots\leq\lambda_k$ arranged in increasing order, and denote the eigenvalues of the distortion matrix $\ol{g}^{-1}\ol{h}$ by $\mu_1\leq\mu_2\leq\cdots\leq \mu_n$ also arranged in increasing order.
	
	For eigenvalue calculations we can assume that the distortion matrix $\ol{g}^{-1}\ol{h}$ is of the block diagonal form
	$$
	\ol{g}^{-1}\ol{h}=\left[
	\begin{array}{cccc}
	g^{-1}h & 0 & 0 & 0\\
	0 & g_2^{-1}h_2 & 0 &0\\
	0 & 0 & \ddots & 0 \\
	0 & 0 & 0 & g_m^{-1}h_m \\
	\end{array}\right].
	$$
	If $\mu_s$ is an eigenvalue of a block of $g_s^{-1}h_s$, then it satisfies the bounds
	$$
	\mu_s\leq \lambda_k^s \mbox{ and } \mu_s\geq \lambda_1^s.
	$$
	Consequently, the eigenvalues $\mu_l$, $l=1,\ldots,n$, of the distortion matrix $\ol{g}^{-1}\ol{h}$ satisfy
	\begin{equation}\label{lambas_and_det}
		\lambda_1^{Q-1}\lambda_k\leq\det(\ol{g}^{-1}\ol{h})\leq \lambda_1\lambda_k^{Q-1}.
	\end{equation}
\end{proposition}
\begin{proof}
	Let $X=(X_1,\ldots,X_n)$ be an adapted frame, and let $(\omega^1,\omega^2,\ldots,\omega^n)$ be its dual frame. The eigenvalues of of the distortion matrix are given by the block diagonal matrix that has $g_s^{-1}h_s$, $s=1,\ldots,m$, as its blocks by the previous proposition. Let us calculate the claimed bounds for the eigenvalues of the blocks $g_s^{-1}h_s$. To ease the calculations, we assume that the local frame $(X_1,\ldots,X_k)$ for $\mD$ is $h$-orthonormal (i.e. orthonormal with respect to $h$). This is justified by the previous proposition on the invariance of the eigenvalues.
	
	Let $s\in \{1,\ldots,m\}$. We will consider the bilinear form $(g_s^{-1})^{\a\b}=g(b^\a,b^\b)$ as a composition of linear mappings. Let us define
	$$
	b:\left(\mD^s/\mD^{s-1}\right)^*\to \otimes^s\mD^*, \ b(\zeta)=\zeta_\a b^\a_{i_1i_2\cdots i_s} \omega^{i_1} \otimes \omega^{i_2}\otimes \cdots \otimes \omega^{i_s}.
	$$
	Here $\zeta=\zeta_\a\omega^\a\in \left(\mD^s/\mD^{s-1}\right)^*$ and the indices run in the ranges $i_l=1,\ldots,k$ and $\a=k_{s-1}+1,\ldots,k_{s}$. 
	
	Let us also define
	$$
	A_g: \otimes^s\mD^*\to \otimes^s\mD, \ A_g(\eta_1 \otimes \eta_2 \otimes \cdots \otimes \eta_s)=g^{-1}\eta_1\otimes g^{-1}\eta_2\otimes\cdots\otimes  g^{-1}\eta_s.
	$$
	The notation $g^{-1}\eta$ means raising the index of given (horizontal) $1$-form $\eta\in\mD^*$:
	$$
	(g^{-1}\eta)^i=g^{ij}\eta_j, \ i,j=1,\ldots k.
	$$
	We define similarly $A_h: \otimes^s\mD^*\to \otimes^s\mD$ for the horizontal metric $h$. With these notations we have that the inverse $g_s^{-1}$ of the $s^{\text{\tiny th}}$ block of $\ol{g}$ can be written as
	\begin{equation}\label{basis_form_for_gs}
		g_s^{-1}=b^TA_g b.
	\end{equation}
	
	Now, let $\mu=\mu_s$ be an eigenvalue of $g_s^{-1}h_s$. That is, there is a vector $0\neq v\in \mD^s/\mD^{s-1}$ such that
	$$
	(g_s^{-1}h_s)v=\mu v.
	$$
	This is equivalent 
	$$
	g_s^{-1}\hat{v}=\mu h^{-1}_s\hat{v},
	$$
	for $\hat{v}=h_sv$.
	Taking (Euclidean) inner product of this equation against $\hat{v}$, we have that 
	\begin{equation}\label{muest}
		\mu=\frac{\langle \hat{v},g_s^{-1}\hat{v}\rangle}{\langle \hat{v},h_s^{-1}\hat{v}\rangle}=\frac{\langle b\hat{v},A_g b\hat{v}\rangle}{\langle b\hat{v},b\hat{v}\rangle}=\frac{\langle b\hat{v},A_g A_h^{-1}b\hat{v}\rangle}{\langle b\hat{v},b
		\hat{v}\rangle}.
	\end{equation}
	Here in the last two equalities, we have used that $h_{ij}=\delta_{ij}$, $i,j=1,\ldots,k$, by the choice that the basis $(X_1,\ldots,X_k)$ for $\mD$ was orthonormal with respect to $h$ (consequently $A_h^{-1}$ is just a $s$-times tensor product of identity matrices).
	It follows that
	$$
	\mu\geq \lambda_{min} (A_gA_h^{-1}) \mbox{ and } \mu\leq \lambda_{max} (A_gA_h^{-1}),
	$$
	where $\lambda_{min}(A_gA_h^{-1})$ and $\lambda_{max}(A_gA_h^{-1})$ are the smallest and the largest eigenvalue of $A_gA_h^{-1}$. 
	
	Since
	$$
	A_gA_h^{-1}
	$$
	is an $s$-times tensor product of matrices $g^{-1}h$ it follows (by the multilinearity of a tensor product) that the eigenvalues of $A_gA_h^{-1}$ are all the products
	$$
	\lambda_{i_1}\lambda_{i_2}\cdots\lambda_{i_s}
	$$
	of the eigenvalues $\lambda_l$, $l=1,\ldots,k$, of $g^{-1}h$. Consequently, we have
	$$
	\mu\geq \lambda_1^s \mbox{ and } \mu\leq \lambda_k^s.
	$$
	That~\eqref{lambas_and_det} holds, follows from these bounds on $\mu$.
\end{proof}

\begin{remark}\label{rmk:step 2}
	We remark that we can actually have better estimates on the eigenvalues of the distortion matrix $\ol{g}^{-1}\ol{h}$. This is due to the fact that the adapted structure constants have symmetries in their lower indices. We will now make this precise in the step $2$ case.
	
	In the step $2$ case, we have that the Popp extensions of $g$ and $h$ have only two blocks in a given adapted frame. The adapted structure constants are \emph{antisymmetric} in their lower indices
	$$
	b_{ij}^\a=-b_{ji}^\a
	$$
	as is obvious from the definition. It follows that we can write the components $g_2^{\a\b}$ of the inverse of $g_2$ as
	$$
	g_2^{\a\b}=2\sum_{I\in \Gamma}\sum_{J\in \Gamma}b^\a_Ia(g)^{IJ}b^\b_{J},
	$$
	where $a(g)^{IJ}\in \R$ is the determinant of a $2\times 2$ submatrix of $g^{-1}$ indicated by the pair of index sets $I, J\in \Gamma=\{(i_1,i_2): i_1\leq i_2\}$. The matrix $a(g)$ defined by the components $a(g)^{IJ}$, $I,J\in \Gamma$, is actually known in other instances as the \textit{compound matrix}. Its eigenvalues are all the products
	$$
	\lambda_{i_1}\lambda_{i_2}
	$$
	where $(i_1,i_2)\in \Gamma$ and $\lambda_{i_j}$, $j=1,2$, is an eigenvalue of $g^{-1}$; see e.g.~\cite{Horn}.
	
	We can assume the adapted frame $(X_1,\ldots,X_k,\ldots,X_n)$ we are using is such that $(X_1,\ldots,X_k)$ is $h$-orthonormal frame for $\mD$.  Thus using similar notations for $h$, we have that
	$$
	h_2^{\a\b}=2\sum_{I\in \Gamma}\sum_{J\in \Gamma}b^\a_I b^\b_{J}
	$$
	and the $a(h)$ has $1$ as its only eigenvalue (of multiplicity $\binom{k}{2}$).
	
	If now $v\in (\mD^2/\mD)^*$ as in~\eqref{muest}, then we have
	$$
	\mu=\frac{\langle v,g_s^{-1}v\rangle}{\langle v,h_s^{-1}v\rangle}=\frac{\langle bv,a(g) bv\rangle}{\langle bv,bv\rangle}=\frac{\langle bv,a(g) a(h)^{-1}bv\rangle}{\langle bv,bv\rangle}.
	$$
	Here the inner product in the latter two quantities is defined by 
	$$\langle Z, W\rangle=\sum_{I\in\Gamma} Z_IW_I.$$
	It follows that
	$$
	\mu\geq \lambda_{min} (a(g)a(h)^{-1}) \mbox{ and } \mu\leq \lambda_{max} (a(g)a(h)^{-1}).
	$$
	Consequently, we have
	$$
	\mu\geq \lambda_1\lambda_2 \mbox{ and } \mu\leq \lambda_{k-1}\lambda_k,
	$$
	where $\lambda_1\leq \lambda_2\leq\cdots\lambda_k$ are the eigenvalues of $g^{-1}h$.
	
	We have arrived better estimates for the eigenvalues in the step $2$ case. In higher steps, the adapted structure constants also have antisymmetries, but we do not pursue here their effect on the eigenvalue estimates of Popp extensions. The enhanced estimates in the step $2$ we achieved above will be used in Section~\ref{subsec:A remark on the definitions of quasiregularity in the Heisenberg groups} when we compare our constructions in the Heisenberg group. 
\end{remark}

The next proposition shows that the Popp measure behaves naturally under pullbacks. This means
$$
f^*P_h=P_{f^*h},
$$
where $f\colon (M,g)\to (N,h)$ is a sufficiently regular \emph{contact} mapping between equiregular subRiemannian manifolds. To be more precise, we can assume that, up to Proposition~\ref{pullbackofPopp}, that $f\in N^{1,Q}_{loc}(M,N)$ is a branched covering with $J_f>0$ a.e. in $M$. Note that under these assumptions, $f$ is contact by Remark~\ref{rmk:section 4} ii).

We include two different regularity assumption for the mapping in question in the next proposition to address a wider audience.
\begin{proposition}\label{pullbackofPopp}Let $M$ and $N$ be two equiregular subRiemannian manifolds with $N$ equipped with a horizontal metric $h$, and let $P_h$ be the Popp volume form on $(N,h)$. Assume that the mapping $f: M\to (N,h)$ is a $C^1$ contact diffeomorphism. Then
	$$
	f^*P_h=P_{f^*h}.
	$$
More generally, if $f\in N^{1,Q}_{loc}(M,N)$ is a branched covering with $J_f>0$ a.e. in $M$, then the above equation holds in the following weaker sense:
$$
\int_Bf^*P_h=\int_BP_{f^*h}
$$
for all balls $B\subset M$.
\end{proposition}

This result can essentially be read from~\cite[Lemma 2.3]{br13}, but we include it here for completeness. Its proof only uses two facts. These are the assumption that the mapping in question maps layers $\mD_M^s$ of $M$ to layers $\mD^s_N$ of $N$ and that the Popp volume form (the Popp measure) induced by the Popp extension is independent of the adapted frame used in the construction. The proof is somewhat lengthy and therefore we record first couple lemmas that will streamline our presentation.
\begin{lemma}\label{forvol1}
	Under the assumptions on the mapping $f$ in Proposition~\ref{pullbackofPopp} the mapping $f$ preserves layers:
	$$
	f_* : \mD_M^s \to \mD_N^s, \quad s=1,\ldots,m.
	$$
	Consequently,
	$$
	f^* : (\mD_N^s)^*\to (\mD_M^s)^*, \quad s=1,\ldots,m.
	$$
\end{lemma}
\begin{proof}
	For $s=1$, this is clear since $f$ is contact. Let then $s=2$ and let $Y\in \mD_M^2$. Thus $Y=[X_1,X_2]+W$, where $X_1,X_2,W\in \mD_M^1$. Since the chain rule holds for $f$, we have by the standard rule of pushforward of the Lie bracket that
	$$
	f_*Y=[f_*X_1,f_*X_2]+f_*W.
	$$
	Since $f$ is contact, we have that $f_*Y\in \mD_N^2$. For general $s=1,\ldots,m$ the claim follows by induction. The latter claim follows by duality.
\end{proof}

\begin{lemma}\label{forvol2}
	Let $p\in M$. Let $X=(X_1,\ldots,X_n)$ and $Y=(Y_1,\ldots, Y_n)$ be adapted frames on neighborhoods of $p$ and $f(p)$, and let $(\omega^1,\ldots,\omega^n)$ and $(\kappa^1,\ldots,\kappa^n)$ be their dual frames. Let 
	$$
	b_{i_1i_2\cdots i_s}^\a \mbox{ and } b_{i'_1i'_2\cdots i'_s}^{\a'}
	$$
	be the corresponding adapted structure constants as before. (Recall that the primed indices are used to indicate quantities with respect to the adapted frame $Y$ near $f(p)$).
	
	Let $f$ be as in Proposition~\ref{pullbackofPopp}. Then with respect to these frames, $f_*$ is a lower block diagonal matrix field we denote by $T$. Let us now denote the diagonal blocks of $T$ by $T_s$, $s=1,\ldots,m$. We set $T_1=T$ for simplicity as before. With respect to the given frames we then have
	$$
	b^{\a'}_{i'_1i'_2\cdots i'_s}=(T^{-T}_s)_{\ph \a}^{\a'} b^{\a}_{i_1 i_2\cdots i_s} \,  T^{\ph i_1}_{i'_1}T^{\ph i_2}_{i'_2}\cdots T^{\ph i_s}_{i'_s}.
	$$
\end{lemma}
\begin{proof}
	By the previous lemma $f_* : \mD_M^s \to \mD_N^s$, $s=1,\ldots,m$, and thus $T$, the representation of $f_*$ with respect to frames $X$ and $Y$, is lower block diagonal. Now a similar calculation to the one in the proof of Proposition~\ref{local_Popp} yields the claim.
\end{proof}

We are now set for the proof of the proposition.

\begin{proof}[Proof of Proposition~\ref{pullbackofPopp}]
	We will only prove the proposition for $f$ being a $C^1$-contact diffeomorphism. The latter claim in Proposition~\ref{pullbackofPopp} follows directly from the proof of the $C^1$-smooth case upon noticing Remark~\ref{rmk:section 4} ii).
	
	The Popp volume form on $(N,h)$ is constructed by orthonormalizing a local adapted coframe by using the Popp extension of the horizontal metric $h$. 
	The construction is independent of the chosen local (co)adapted frame yielding the Popp volume form globally on $N$, see e.g.~\cite{br13, m02}. 
	
	Let $p\in M$ and let $X$ be an adapted frame near $p$ and let $\omega$ be its dual frame. Let $\theta$ by the orthonormal adapted coframe that defines the Popp volume form near $f(p)$, and let $Y$ by its adapted dual frame. This particular choice of the frames is just to simplify the needed calculations below.
	The Popp volume form $P_h$ reads
	$$
	P_h=\theta^1\wedge\cdots\wedge\theta^k\wedge\cdots\wedge \theta^n.
	$$
	We need to show that the pull back of the Popp volume form by $f^*$ is a wedge product of an $\ol{f^*h}$-orthonormal coframe.
	
	We divide the coframe $\omega$ into parts by setting
	$$
	\omega(s)=(\omega^{k_{s-1}+1},\ldots \omega^{k_s}), \quad s=1,\ldots, m.
	$$
	Moreover, we denote
	$$
	\wedge\omega(s)=\omega^{k_{s-1}+1}\wedge\ldots \wedge\omega^{k_s}.
	$$
	Let us present the pushforward $f_*$ by a matrix field $T$ defined with respect to the adapted frames $X$ and $Y$. 
	
	By the Lemma~\ref{forvol1} we have that $f_*$ preserves layers and thus $T$ is lower block triangular with respect to $X$ and $Y$. Denote the diagonal blocks of $T$ by $T_s$ and note that the determinant of $T$ is the product of the determinants of the blocks $T_s$. By the standard rules of wedge products we have 
	\begin{align*}
	f^*P_h&=\det(T)\omega^1\wedge\cdots\wedge\omega^n=\det(T_1)\cdots \det(T_m)\omega^1\wedge\cdots\wedge\omega^n\\
	&=(\det(T_1)\wedge\omega(1))\wedge(\det(T_2)\wedge\omega(2))\wedge\cdots\wedge(\det(T_m)\wedge\omega(m)).
	\end{align*}
	Furthermore, we have
	$$
	\det(T_s)\wedge\omega(s)=\det(T_s)\omega^{k_{s-1}+1}\wedge\ldots \wedge\omega^{k_s}=T_s^T\theta^{k_{s-1}+1}\wedge T_s^T\theta^{k_{s-1}+2}\wedge\cdots\wedge T_s^T\theta^{k_s}.
	$$
	
	Let $s=1,\ldots,m$. We need to show that for any $\a',\b'\in \{k_{s-1}+1,\ldots,k_s\}$ the inner product of $T_s^T\theta^{\a'}$ and $T_s^T\theta^{\b'}$ is $\delta^{\a'\b'}$. We have
	$$
	(T_s^T)\theta^{\a'}=(T_s^T)^{\a'}_{\ph \a} \omega^{\a},
	$$
	where the summation is over $\a\in\{k_{s-1}+1,\ldots,k_s\}$ as usual. The inner product for fixed $\a'$ and $\b'$ is
	$$
	\ol{f^*h}\left((T^T_s)\theta^{\a'},(T^T_s)\theta^{\b'}\right)=(T_s^T)^{\a'}_{\ph \a} (T_s^T)^{\b'}_{\ph \b} \left\langle \omega^{\a}, b_M^T A
	_{f^*h} b_M \omega^{\b}\right\rangle.
	$$
	Now, we have by definition
	$$
	\left\langle \omega^{\a}, b_M^T A
	_{f^*h} b_M \omega^{\b}\right\rangle=b^{\a}_{i_1\cdots i_s}(f^*h)^{i_1j_1}\cdots(f^*h)^{i_sj_s}b^{\b}_{j_1\cdots j_s}.
	$$
	Since the horizontal metric transforms tensorially, we have
	$$
	(f^*h)^{i_lj_l}=(T^{-1})^{i_l}_{\ph i'_l}(T^{-1})_{\ph j'_l}^{ j_l}h^{i'_lj'_l}, \quad l=1,\ldots,s.
	$$
	Now, by Lemma~\ref{forvol2}, we have
	$$
	b^{\a'}_{i'_1i'_2\cdots i'_s}=(T^{-T}_s)_{\ph \a}^{\a'} b^{\a}_{i_1 i_2\cdots i_s} \,  T^{\ph i_1}_{i'_1}T^{\ph i_2}_{i'_2}\cdots T^{\ph i_s}_{i'_s}
	$$
	Plugging all these in, we have
	\begin{align*}
	\ol{f^*h}\left((T_s)\theta^{\a'},(T_s)\theta^{\b'}\right)&=(T_s^T)^{\a'}_{\ph \a} (T_s^T)^{\b'}_{\ph \b}b^{\a}_{i_1\cdots i_s}(f^*h)^{i_1j_1}\cdots(f^*h)^{i_sj_s}b^{\b}_{j_1\cdots j_s} \\ 
	&=b^{\a'}_{i'_1\cdots i'_s}h^{i'_1j'_1}\cdots h^{i'_sj'_s}b^{\b'}_{j'_1\cdots j'_s}=\ol{h}(\theta^{\a'},\theta^{b'}) \\
	&=\delta^{\a'\b'}
	\end{align*}
	as requested.
\end{proof}

\begin{remark}\label{rmk:Riemannian volume of Popp extension=Popp volume of horizontal}
	The Popp volume form $P_g$ of a horizontal metric $g$ is locally the Riemannian volume form $dV_{\ol{g}}$ of a local Popp extension $\ol{g}$ of $g$. We consider here $\ol{g}$ as a local Riemannian metric. This fact does not depend on the adapted basis used for the construction of the local Popp extension. 
	
	To see this, let $(E_1,E_2,\ldots,E_n)$ be any $\ol{g}$-orthonormal basis, then the Riemannian volume form is defined by the requirement
	$$
	dV_{\ol{g}}(E_1,E_2,\ldots,E_n)=1
	$$
	for every local oriented frame.
	See e.g.~\cite{l97}. If we take $(E_1,E_2,\ldots,E_n)$ to be an adapted frame that is $\ol{g}$-orthonormal, then 
	$$
	P_g(E_1,E_2,\ldots,E_n)=1.
	$$
	Changing to another $\ol{g}$-orthonormal frame, with same orientation, amounts multiplying by the determinant of the change of basis matrix, which of course is just $1$. 
	Thus 
	$$
	P_g=dV_{\ol{g}}.
	$$
\end{remark}

In Section~\ref{sec:Quasiregular mappings}, the volume Jacobian of $f$ was defined to be the volume derivative of the pullback measure $f^*\Vol_N$ with respect to the original measure $\Vol_M$ on $M$. Alternatively, we can also think of the Jacobian as the representation function of the pullback $n$-form $f^*P_h$ with the $n$-form $P_g$ on $M$. We denote by $J(\cdot,f)$ this representation function, i.e., $J(x,f)P_g(x)=f^*P_h(x)$. 
\begin{lemma}\label{lemma:coincidence of two Jacobian}
	Let $f:(M,g)\to (N,h)$ be a branched covering in $N^{1,Q}_{loc}(M,N)$ such that $J_f>0$ a.e. in $M$. Then $J_f=(\det(\ol{g}^{-1}\ol{f^*h}))^{1/2}$ a.e. in $M$.
\end{lemma}
\begin{proof}
	By Proposition~\ref{pullbackofPopp}, we have that $f^*P_h=P_{f^*h}$ in the weak sense. It follows that
	\begin{align*}
		J_f(x_0)&=\frac{d(f^*\Vol_N)}{d\Vol_M}(x_0)=\lim_{r\to 0}\frac{\int_{B(x_0,r)}f^*P_{h}}{\int_{B(x_0,r)}P_g}\\
		&=\lim_{r\to 0}\frac{\dashint_{B(x_0,r)}P_{f^*h}}{\dashint_{B(x_0,r)}P_g}=\left(\det(\ol{g}^{-1}\ol{f^*h}(x_0))\right)^{1/2}
	\end{align*}
	for a.e. $x_0\in M$, by the Lebesgue differentiation theorem.
\end{proof}

\subsection{Equivalence of the horizontal and the subRiemannian quasiregularity}\label{sec:Equivalence of the horizontal and the subRiemannian quasiregularity}

We are now ready to define a distortion function $K^2(\ol{g},\ol{h})$ of two Popp extensions $\ol{g},\ol{h}$. Motivated by~\cite{l14}, and also~\cite{d99}, we set
\begin{equation}\label{distortion}
	K^2(\ol{g},\ol{h})=\frac{||g^{-1}h||^Q}{\det(\ol{g}^{-1}\ol{h})},
\end{equation}
where $||\cdot||$ is the sup-norm\footnote{ To avoid confusions with~\cite{l14}, we remark that the work~\cite{l14} uses the Hilbert-Schmidt -norm instead.}. Note that in the numerator of~\eqref{distortion} only the horizontal metrics are used. 

We introduce the \emph{subRiemannian definition} of quasiregularity.
\begin{definition}(SubRiemannian $K$-quasiregular mappings)\label{def:subRiemannian qr}
	Let $f\colon (M,g)\to(N,h)$ be a branched covering between equiregular subRiemannian manifolds $(M,g)$ and $(N,h)$. We say the mapping $f$ is a \textit{subRiemannian $K$-quasiregular mapping} if $f\in N^{1,Q}_{loc}(M,N)$, $J_f>0$ a.e. in $M$ and 
	$$
	||g^{-1}f^*h||^Q\leq K^2 \det(\ol{g}^{-1}\ol{f^*h}) \mbox{ a.e. in }M
	$$
	or equivalently
	$$
	K^2(\ol{g},\ol{f^*h})\leq K^2 \mbox{ a.e. in }M.
	$$
\end{definition}
We remark that the square of $K$ appearing in the definition is due to fact that we use pull back of the horizontal tensor $h$ is quadratic in the differential of $f$. See also~\cite{l14}. 

Here $\ol{f^*h}$ is the Popp extension of the pullback metric $f^*h$. The pullback $f^*h$ is well defined since by Proposition~\ref{prop:on differentiability}, the mapping $f$ is $P$-differentiable a.e. in $M$. The norm $||g^{-1}f^*h||$ is the sup-norm of $g^{-1}f^*h$. We remark that the horizontal $\binom{1}{1}$-tensor $g^{-1}f^*h$ is the same as 
$$
Df^*Df : (\mD_M,g) \to (\mD_N,h),
$$
where the adjoint $Df^*$ is defined as the adjoint of $Df$ between inner product spaces $(\mD_M,g)$ and $(\mD_N,h)$, cf.~\cite{l14}. We can also regard $g^{-1}f^*h$ as an element of $\mbox{End}(\mD_M)$ implying that the eigenvalue problem for $g^{-1}f^*h$ is well defined, i.e., independent of the choice of basis for $\mD_M$.
\begin{remark}\label{rmk:on subRiemannian quasiregularity}	
	The a priori assumption $J_f>0$ a.e. in $M$ is only used to ensure that $\ol{g}^{-1}\ol{f^*h}$ is well-defined. The reason is that we need the horizontal $2$-form $f^*h$ to be positive definite a.e. to construct $\ol{f^*h}$. Equivalently, we could assume that the pushforward $f_*:\mD_M\to\mD_N$ is injective a.e. in $M$. 
\end{remark}

We define \emph{horizontal distortion} to be
$$
H^2(g,h)=\frac{||g^{-1}h||^k}{\det(g^{-1}h)}.
$$
Horizontally $K$-quasiregular mappings are defined as follows (compare it with~\cite{l14}).
\begin{definition}(Horizontally $K$-quasiregular mappings)\label{def:horizontal qr}
	Let $f\colon (M,g)\to(N,h)$ be a branched covering between equiregular subRiemannian manifolds $(M,g)$ and $(N,h)$.  We say that the mapping $f$ is \textit{horizontally $K$-quasiregular} if $f\in N^{1,Q}_{loc}(M,N)$ and it satisfies
	$$
	||g^{-1}f^*h||^k\leq K^2 \det(g^{-1}f^*h) \mbox{ a.e. in } M
	$$
	or equivalently
	$$
	H^2(g,f^*h)\leq K^2 \mbox{ a.e. in } M.
	$$
\end{definition}
We remark that if the horizontal distortion is equal to $1$ a.e., we have that
$$
f^*h=c g,
$$
for some a.e. positive function. This can be seen by an elementary argument using inequality of geometric and arithmetic means to eigenvalues of $g^{-1}f^*h$ (cf.~\cite{l14}).

The next proposition shows the interrelation between the introduced distortions. In particular it implies that, when either of the distortions is one, so is the other.

\begin{proposition}\label{prop:compare distortions}
	Let $g$ and $h$ be horizontal metrics on an equiregular subRiemannian manifold, and let $\ol{g}$ and $\ol{h}$ be their Popp extensions. Then
	\begin{align*}
		H^2(g,h)\leq K^2(\ol{g},\ol{h})\leq \big(H^2(g,h)\big)^{Q-1}.
	\end{align*}
	
\end{proposition}
\begin{proof}
	Let us denote the eigenvalues of $g^{-1}h$ by $\lambda_1\leq \lambda_2\leq \cdots\leq \lambda_k$ arranged in increasing order, and let us denote the eigenvalues of $\ol{g}^{-1}\ol{h}$ by $\mu_1\leq \mu_2\leq\cdots \leq\mu_n$ also arranged in increasing order. 
	
	We first prove the latter inequality. Note that we have
	$$
	\frac{\lambda_k}{\lambda_1}\leq \frac{\lambda_k}{\lambda_1}\frac{\lambda_k}{\lambda_2}\cdots \frac{\lambda_k}{\lambda_k}=H^2(g,h).
	$$
	Now 
	$$
	K^2(\ol{g},\ol{h})=\frac{\lambda_k^Q}{\det(\ol{g}^{-1}\ol{h})}\leq \frac{\lambda_k^Q}{\lambda_1^{Q-1}\lambda_k}=\left(\frac{\lambda_k}{\lambda_1}\right)^{Q-1}\leq \big(H^2(g,h)\big)^{Q-1},
	$$
	where we have used~\eqref{lambas_and_det} in the first inequality.
	
	For the first inequality, note that we have
	\begin{equation}\label{fromKtoH}
		H^2(g,h)=\frac{\lambda_k^k}{\lambda_1\cdots\lambda_k}=\frac{\lambda_k^{k+2n_2+3n_3+\cdots +mn_m}}{\lambda_1\cdots\lambda_k \lambda_k^{2n_2+3n_3+\cdots + mn_m}} \leq \frac{\lambda_k^Q}{\lambda_1\cdots\lambda_k\mu_{k+1}\mu_{k+2}\cdots\mu_n}.
	\end{equation}
	Here in the inequality we have used the bounds
	$$
	\mu_s\leq \lambda_k^s
	$$
	from Proposition~\ref{eigenval_compr}. But now the right hand side of~\eqref{fromKtoH} is just $K^2(\ol{g},\ol{h})$ proving the other part of the claim. 
\end{proof}

We now prove the equivalence of the two introduced definitions of quasiregularity.
\begin{theorem}\label{thm:equivalence of horizontal and subRiemannian}
	Let $(M,g)$ and $(N,h)$ be equiregular subRiemannian manifolds. If $f\colon M\to N$ is a subRiemannian $K$-quasiregular mapping, then it is horizontally $K$-quasiregular. Conversely, if $f:M\to N$ is a horizontally $K$-quasiregular mapping, then it is subRiemannian $K^{Q-1}$-quasiregular.
\end{theorem}
\begin{proof}
	If $f\colon M\to N$ is a subRiemannian $K$-quasiregular mapping. Then it follows directly from Proposition~\ref{prop:compare distortions} that $f$ is horizontally $K$-quasiregular. 
	
	For the reverse direction, we only need to show that if $f\colon M\to N$ is a horizontally $K$-quasiregular mapping, then $J_f>0$ a.e. in $M$. Since this is a local property, by Lemma~\ref{lemma:positiveness of Jacobian}, it suffices to show that $f$ is locally analytically quasiregular. Let us denote the eigenvalues of $g^{-1}f^*h$ by $\lambda_1\leq \lambda_2\leq \cdots\leq \lambda_k$ arranged in increasing order. Then the horizontal $K$-quasiregularity implies that
	\begin{align*}
		\lambda_k^k\leq K\lambda_1\cdots\lambda_k,
	\end{align*}
	from which it follows in particular that
	\begin{align*}
		\lambda_k\leq K\lambda_1.
	\end{align*}
	On the other hand, by Lemma~\ref{lemma:bounds on Jacobian}, Lemma~\ref{lemma:Lip=norm of differential}, Proposition~\ref{prop:minimal upper gradient} and the above inequality, we have for each $x\in M$, there exist a radius $r_x>0$ and a positive constant $C>0$ such that

	\begin{align*}
		(g_f(x_0))^{2Q}&=\lambda_k^Q\leq K^Q\lambda_1^Q\leq CK^Q(J_f(x_0))^2\quad \text{for a.e. }  x_0\in B(x,r_x).
	\end{align*}
	This implies that $f$ is analytically $CK^Q$-quasiregular in $B(x,r_x)$. The claim follows again from Proposition~\ref{prop:compare distortions}.

\end{proof}

\section{Proofs of the main result}\label{sec:Proofs of the main result}
\subsection{Equivalence of the analytic and the geometric quasiregularity}
It is a well-known fact, due to Williams~\cite{w12proc}, that for a homeomorphism between fairly general metric measure spaces, the analytic $K$-quasiconformality and the geometric $K$-quasiconformality are equivalent. The proof there also works in the quasiregular case, but for the convenience of the readers, we provide the details here. In the following we use notions and definitions introduced in Section~\ref{sec:Calculus on subRiemannian manifolds}.

\begin{proposition}\label{prop:analytic equal geometric}
	The mapping $f:M\to N$ is analytically $K$-quasiregular if and only if it is geometrically $K$-quasiregular. 
\end{proposition}
\begin{proof}
	i) (Analytic implies geometric)
	
	Note first that since $f\in N^{1,Q}_{loc}(M,N)$ it follows that $f$ is absolutely continuous on every $\gamma\in\Gamma':=\Gamma\backslash \Gamma_0$, where $\Modd_Q(\Gamma_0)=0$ (see e.g.~\cite[Section 6.3]{hkst12}).
	
	Fix an open set $\Omega\subset M$ and a curve family $\Gamma$ in $\Omega$. Let $\rho:N\to [0,\infty]$ be a test function for $f(\Gamma)$. 
	Define $\rho':\Omega\to [0,\infty]$ by setting
	\begin{align*}
	\rho'(x)=(\rho\circ f)(x)g_f(x).
	\end{align*}
   Thus,
	\begin{align*}
	\int_{\gamma}\rho' ds=\int_\gamma (\rho\circ f)g_f ds\geq \int_{f\circ \gamma}\rho ds\geq 1
	\end{align*}
	for every $\gamma\in \Gamma'$. Since $\rho'$ is a Borel function, we conclude that $\rho'$ is a test function for $\Gamma'$.
	
	Since $f$ is analytically $K$-quasiregular, $g_f(x)^Q\leq KJ_f(x)$ for a.e. $x\in \Omega$. By~\cite[Theorem B]{gnw15} and the subadditivity of modulus, 
	\begin{align*}
	\Modd_Q(\Gamma)&=\Modd_Q(\Gamma')\leq \int_{\Omega}\rho'(x)^Qd\Vol_M(x)=\int_\Omega \rho(f(x))^Qg_f(x)^Qd\Vol_M(x)\\
	&\leq K\int_\Omega \rho(f(x))^QJ_f(x)d\Vol_M(x)=K\int_NN(y,f,\Omega)\rho(y)^Qd\Vol_N(y).
	\end{align*}
	
	ii)  (Geometric implies analytic)
	
	The proof is a modification of the one used in the quasiconformal case~\cite{w12proc}. 
	By Theorem~\ref{thm:characterization of Sobolev spaces}, we only need to verify~\eqref{eq:mod for Sob regularity}. First of all, note that whenever $U\subset M$ is an open set such that $f^*\Vol_N(U)<\infty$, we have 
	\begin{align*}
	\liminf_{\varepsilon\to 0}\varepsilon^Q\Modd_Q\Big(f^{-1}\big(\mathcal{C}_\varepsilon(f(U))\big)\Big)\leq Kf^*\Vol_N(U)<\infty,
	\end{align*}
	where $\mathcal{C}_\varepsilon(\cdot)$ is the curve family defined as in Section~\ref{subsec:Sobolev spaces} (before Theorem~\ref{thm:characterization of Sobolev spaces}). Indeed, given such $\varepsilon>0$ and $U\subset M$, the function $\rho=\varepsilon^{-1}\,\chi_{f(U)}$ is admissible for $f^{-1}\big(\mathcal{C}_\varepsilon(f(U))\big)$. Thus the $K_O$-inequality (cf.~Definition~\ref{def:geometric def for qr})  implies that 
	\begin{align*}
	\varepsilon^Q\Modd_Q\Big(f^{-1}\big(\mathcal{C}_\varepsilon(f(U))\big)\Big)\leq K\varepsilon^Q\int_N\varepsilon^{-Q}N(y,f,U)d\Vol_N(y)=Kf^*\Vol_N(U).
	\end{align*}
	Theorem~\ref{thm:characterization of Sobolev spaces} gives $f|_{U}\in N^{1,Q}(U,N)$ and
	\begin{align*}
	\int_U g_{f|_U}^Qd\Vol_M\leq Kf^*\Vol_N(U).
	\end{align*}
	It follows that $f\in N^{1,Q}_{loc}(M,N)$ and that
	\begin{align*}
	\int_U g_f^Qd\Vol_M\leq Kf^*\Vol_N(U)
	\end{align*}
	for every open set $U\subset M$. Note that the measure $f^*\Vol_N$ is Borel, the above inequality holds for all Borel sets as well. The claim follows from the definitions of $f^*\Vol_N$ and volume Jacobian $J_f$. 	
\end{proof}

\subsection{Proof of the main theorem, Theorem~\ref{thm:equivalence of quasiregular mappings}}
\begin{proof}[Proof of Theorem~\ref{thm:equivalence of quasiregular mappings}]
(A).  1) implies 2).

This follows directly from the definitions.
	
(B).  2) implies 4).
	Let $\lambda_1,\ldots,\lambda_k$ be the eigenvalues of $g^{-1}f^*h$ arranged in increasing order. We have 
	$$
	\|Df(x)\|=\lambda_k^{1/2} \mbox{ and } \|Df(x)\|_s=\lambda_1^{1/2}.
	$$
	(Of course the eigenvalues $\lambda_i$ depend on $x$, but we omit $x$ from the notation.)
	Since $f$ is weak metrically $H$-quasiregular, by Lemma~\ref{Lip=norm of differential} and Proposition~\ref{prop:Jacobian=det Df}, we have for a.e. $x\in M$, 
	\begin{align*}
	\Lip f(x)=\|Df(x)\|\quad \text{ and }\quad J_f(x)=J_{Df(x)}(0).
	\end{align*}
	On the other hand, by Lemma~\ref{lemma:simple}, for a.e. $x\in M$,
	\begin{align*}
	\frac{\|Df(x)\|}{\|Df(x)\|_s}\leq H.
	\end{align*}
	It follows from the above and Proposition~\ref{eigenval_compr} that for a.e. $x\in M$
	\begin{align*}
	\Lip f(x)^Q&=\|Df(x)\|^Q\leq H^{Q-1}\|Df(x)\|\|Df(x)\|_s^{Q-1}\\
	&\leq H^{Q-1}J_{Df(x)}(0)=H^{Q-1}J_f(x).
	\end{align*}
	This implies that $f$ is analytically $K$-quasiregular with $K=H^{Q-1}$.

(C). 4) implies 6). 
	
	If $f$ is analytically $K$-quasiregular, then by Lemma~\ref{lemma:positiveness of Jacobian} we have that $J_f>0$ a.e. in $M$. Moreover, it follows from Proposition~\ref{prop:minimal upper gradient} that $g_f=\Lip f$ and so by Lemma~\ref{lemma:Lip=norm of differential} and Lemma~\ref{lemma:coincidence of two Jacobian}, we have 
	\begin{align*}
		||g^{-1}f^*h(x)||^{Q/2}=\|Df(x)\|^{Q}\leq KJ_f(x)=K\det(\overline{g}^{-1}\overline{f^*h})^{1/2}
	\end{align*}  
	for a.e. $x\in M$.
	
(D). 6) implies 4).

This is essentially the same as in (C). By Proposition~\ref{prop:minimal upper gradient}, Lemma~\ref{lemma:Lip=norm of differential} and Lemma~\ref{lemma:coincidence of two Jacobian}, we have
    	\begin{align*}
    	g_f(x)^Q=\|Df(x)\|^{Q}=||g^{-1}f^*h(x)||^{Q/2}\leq K\det(\overline{g}^{-1}\overline{f^*h})^{1/2}=KJ_f(x)
    	\end{align*}  
   for a.e. $x\in M$.
	
(E). 6) implies 1).
	
	The fact that $H_f(x)<\infty$ for all $x\in M$ follows from 3) and the fact that equiregular subRiemannian manifolds of homogeneous dimension $Q\geq 2$ have locally $Q$-bounded geometry (see e.g.~\cite{gnw15,mm95}). Indeed, we may use the proof of Theorem 7.1 from~\cite{or09} to conclude that $H_f(x)\leq H(K,i(x,f))$ for each $x\in M$, where $i(x,f)$ is the local index of $f$ at $x$; see~\cite[Remark 5.14]{gnw15} and also~\cite[Theorem 1.3]{w12}. Thus we are left to show that $H_f(x)\leq K$ a.e. in $M$.
	
	Note that 
	\begin{align*}
	H_f(x)=\limsup_{r\to 0}\frac{L_f(x,r)}{l_f(x,r)}=\limsup_{r\to 0}\frac{L_f(x,r)/r}{l_f(x,r)/r}\leq \frac{L_f(x)}{l_f(x)},
	\end{align*}
	where 
	\begin{align*}
	L_f(x)=\limsup_{r\to 0}\frac{L_f(x,r)}{r}\quad \text{and}\quad l_f(x)=\liminf_{r\to 0}\frac{l_f(x,r)}{r}.
	\end{align*}
	Since $f$ is $P$-differentiable a.e. in $M$, then we have by~\eqref{eq:def maximal norm} and~\eqref{eq:def for minimal norm} that $L_f(x)=\|Df(x)\|$ and $l_f(x)=\|Df(x)\|_s$ a.e. in $M$.
	
   Since $f$ is subRiemannian $K$-quasiregular, by Proposition~\ref{eigenval_compr}, we have at a.e. $x\in M$
   \begin{align*}
   	\|Df(x)\|^Q\leq K\det(\overline{g}^{-1}\overline{f^*h})^{1/2}\leq K\|Df(x)\|_s\|Df(x)\|^{Q-1}
   \end{align*}
   and so at each such $x\in M$,
   \begin{align*}
   H_f(x)\leq \frac{L_f(x)}{l_f(x)}=\frac{\|Df(x)\|}{\|Df(x)\|_s}\leq K.	
   \end{align*}

(F). 2) implies 1).

Note that the implications 2)$\Rightarrow$4)$\Rightarrow$6)$\Rightarrow$1) do not give us the claimed relation between the quasiregularity constants as claimed in Theorem~\ref{thm:equivalence of quasiregular mappings}. We thus need a separate argument to show that the weak metric quasiregularity implies the metric quasiregularity with the exact same constant.

The fact that $H_f(x)<\infty$ for all $x\in M$ follows from the fact that 2) implies 4) and that 4) implies 1). It remains to show that $H_f(x)\leq H$ for a.e. $x\in M$. This follows readily from the proof of Proposition 5.22 in~\cite{gnw15}. Indeed, the proof of Proposition 5.22 there implies that if $x\in M$ is a point such that $f$ is $P$-differentiable at $x$ and that $J_f(x)>0$, then we have
\begin{align*}
	H_f(x)=\limsup_{r\to 0}\frac{L_f(x,r)}{l_f(x,r)}=\lim_{r\to 0}\frac{L_f(x,r)}{l_f(x,r)}=\liminf_{r\to 0}\frac{L_f(x,r)}{l_f(x,r)}=h_f(x). 
\end{align*}

(G). We compute the remaining dependences on the quasiregularity constants. We first show that if $f$ is metrically $H$-quasiregular, then it is horizontally $H^{k-1}$-quasiregular. 
It follows from Lemma~\ref{lemma:simple} that the metric $H$-quasiregularity implies that $\lambda_k\leq H^2\lambda_1$. It follows that
\begin{align*}
	\|g^{-1}f^*h\|^{k/2}=\lambda_k^{k/2}\leq H^{k-1}(\lambda_1^{k-1}\lambda_k)^{1/2}\leq H^{k-1}\det(g^{-1}f^*h)^{1/2}.
\end{align*}

We next show that if $f$ is horizontally $K_0$-quasiregular, then it is metrically $K_0$-quasiregular. Note that, we already know from the preceding equivalences that $H_f(x)<\infty$ for all $x\in M$, and thus, we only need to verify that $H_f(x)\leq K_0$ for a.e. $x\in M$. 
The horizontal $K_0$-quasiregularity implies that $\lambda_k^k\leq K_0^2\lambda_1\cdots\lambda_k$. In particular, it follows that $\lambda_k\leq K_0^2\lambda_1$. We may repeat the argument in (E) to conclude that at a.e. $x\in M$
\begin{align*}
	H_f(x)\leq \frac{\lambda_k^{1/2}}{\lambda_1^{1/2}}\leq K_0.
\end{align*}

The final implication iv) is included in Theorem~\ref{thm:equivalence of horizontal and subRiemannian}.
\end{proof}

\begin{remark}\label{rmk:on the main proof}
	i) In both (B) and (E) of the above proof, our reasoning relies crucially on Proposition~\ref{eigenval_compr}, where the estimates on the eigenvalues are sharper than that provided in Lemma~\ref{lemma:bounds on Jacobian}.
	
	ii) If we were able to show that the constant $C$ from Lemma~\ref{lemma:bounds on Jacobian} can be chosen to be one. Then the above proof (with $K=1$) would directly imply the equivalences of 1), 2), 4) and 5) in Corollary~\ref{coro:equivalence of 1-quasiregular mappings}.
\end{remark}

\subsection{A remark on the definitions of quasiregularity in the Heisenberg groups}\label{subsec:A remark on the definitions of quasiregularity in the Heisenberg groups}
In this section, we point out that for mappings on Heisenberg groups $\mathbb{H}^n$, the horizontal definition of quasiregularity, Definition~\ref{def:horizontal qr}, is quantitatively equivalent with a definition of quasiregularity introduced by Dairbekov~\cite{d99,d00}.

We first recall Dairbekov's definition. Let $U\subset \mathbb{H}^n$ be a domain and let $f\colon U\to \mathbb{H}^n$ be a continuous mapping. Dairbekov defines~\cite[Definition 1.3]{d00} the mapping $f$ to be $K$-quasiregular  
if the following three conditions are satisfied:
    \begin{itemize}
		\item $f\in HW^{1,Q}_{loc}(U,\mathbb{H}^n)$,
		\item $f$ is contact,
		\item $\|D_Hf(x)\|^Q\leq KJ(x,f)$ \text{ almost everywhere in }U.
	\end{itemize}
In this case we say that $f$ is $K$-quasiregular \textit{in the sense of Dairbekov}. The measure on $\mathbb{H}^n$ is the $\R^{2n+1}$ Lebesgue measure and $Q=2n+2$ for the Heisenberg group $\mathbb{H}^n$. Note that the Lebesgue $(2n+1)$-measure coincides with the Popp measure $\Vol_Q$ on $\mathbb{H}^n$.

Above, $HW^{1,Q}_{loc}(U,\mathbb{H}^n)$ denotes the \textit{local horizontal Sobolev space of mappings} $f:U\to \mathbb{H}^n$. These are defined component-wise as follows. On the Heisenberg group we can express $f$ in the global coordinates of $\R^{2n+1}$ as $f=(f_1,\dots,f_{2n+1})$. We define that $f$ belongs to $HW^{1,Q}_{loc}(U,\mathbb{H}^n)$ 
if each component $f_i$, with $i=1,\dots,2n+1$, of $f$ belongs to the $HW^{1,Q}_{loc}(U)$. Here $HW^{1,Q}_{loc}(U)$ denotes the \textit{local horizontal Sobolev space of functions} 
$$
HW^{1,Q}_{loc}(U)=\{u:U\to \mathbb{H}^n: u\in L^Q_{loc}(U), \quad  X_iu \in L^Q_{loc}(U), \ i=1,\dots,2n \} 
$$
See~\cite[Section 11]{hk00} for more details on these spaces.

Let us first make the observation that a continuous mapping $f\colon U\to \mathbb{H}^n$ belongs to $HW^{1,Q}_{loc}(U,\mathbb{H}^n)$ if and only if $f\in N^{1,Q}_{loc}(U,\mathbb{H}^n)$. To see this, notice first that in the Heisenberg groups $\mathbb{H}^n$ a continuous mapping $f\in N^{1,Q}_{loc}(U,\mathbb{H}^n)$ if and only if each of its components $f_i$, $i=1,\dots,2n+1$, belongs to $N^{1,Q}_{loc}(U)$. Our claim then follows from the well-known fact that \textit{a continuous function $u\in HW^{1,Q}_{loc}(U)$ if and only if it belongs to $N^{1,Q}_{loc}(U)$}; see~\cite[Section 11]{hk00}. 

Let $f\in HW^{1,1}_{loc}(U,\mathbb{H}^n)$. The vectors $X_if(x)=(X_if_1(x),\dots,X_if_{2n+1}(x))$, $i=1,\dots,2n$, viewed as column vectors can be treated as tangent vectors at $f(x)$, i.e. $X_if(x)\in T_{f(x)}\mathbb{H}^n$. A mapping $f:U\to \mathbb{H}^n$ of the class $HW^{1,1}_{loc}(U,\mathbb{H}^n)$ is (weakly) \textit{contact} if for each $i=1,\dots,2n$, the tangent vector $X_if(x)$ belongs to the horizontal tangent space $HT_{f(x)}\mathbb{H}^n$. The \textit{formal horizontal differential} $D_Hf(x):HT_xU\to HT_{f(x)}\mathbb{H}^n$ of a contact mapping $f:U\to \mathbb{H}^n$ is defined by its action on the basis vectors:
\begin{align*}
	D_Hf(x)X_i=X_if(x)\quad i=1,\dots,2n.
\end{align*}

Let $(X_1,\ldots,X_{2n})$ be standard global frame for the horizontal distribution satisfying the commutation relations
$$
[X_j,X_{j+n}]=-4T, \quad j=1,\ldots,n.
$$
The other commutators are identically zero. The \textit{horizontal Jacobian} $HJ(x,f)$ of $f$ is defined as the determinant of $D_Hf(x)$ expressed in the frame $(X_1,\dots,X_{2n})$ (used for both $HT_xU$ and $HT_{f(x)}\mathbb{H}^n$). The formal horizontal differential can be extended linearly as a homomorphism of the Heisenberg algebra so that one can talk about the \textit{formal differential}. Then the \textit{full Jacobian} $J(x,f)$ of $f$ is defined to be the determinant of the formal differential expressed in the full frame $(X_1,\ldots,X_n,T)$ of $T\mathbb{H}^n$. A computation~\cite{d00} shows that 
\begin{equation}\label{rel_of_jacobs}
J(x,f)=HJ(x,f)^{\frac{n+1}{n}}. 
\end{equation}

As in Section~\ref{subsec:Some basic facts about quasiregular mappings in subRiemannian manifolds}, the norm of the formal horizontal differential of a contact mapping $f$ is defined to be
\begin{align*}
	\|D_Hf(x)\|:=\max_{v} \Big\{|D_Hf(x)v|_g:v\in HT_xU\ \text{and}\ |v|_g=1\Big\}. 
\end{align*}
The norms here are defined with respect to a horizontal metric $g$, which itself is defined by declaring the frame $(X_1,\ldots X_n)$ orthonormal.

Recall here that the branch set $\mathcal{B}_f$ of a continuous mapping $f\colon M\to N$ consists of those points $p$ in $M$, where $f$ fails to be a local homeomorphism at $p$.

To compare the regularity assumptions in the Dairbekov's definition and our horizontal definition, we need the following simple lemma, which is certainly known to experts.
\begin{lemma}\label{lemma:continuity}
    Let $f\colon U\to \mathbb{H}^n$ be an analytically quasiregular mapping. Then $f\in N^{1,p}_{loc}(U,\mathbb{H}^n)$ for some $p>Q$ and hence $f$ is locally $\alpha$-H\"older continuous with $\alpha=1-\frac{Q}{p}$.
\end{lemma}
\begin{proof}
 Our aim is to show that $J_f\in L^p_{loc}(U)$ for some $p>Q$. For this, we fix a point $x\in \Omega\backslash f^{-1}\big(f(\mathcal{B}_f)\big)$  and
 a subdomain $\mathcal{K}\subset\subset U$ with $x\in \mathcal{K}$. Note that $f|_{K\backslash f^{-1}\big(f(\mathcal{B}_f)\big)}$ is a covering mapping and so we may write
\begin{equation*}
    K\backslash f^{-1}\big(f(\mathcal{B}_f)\big)=\bigcup_{i=1}^kU_i,
\end{equation*}
where $U_i\cap U_j=\emptyset$ if $i\neq j$ and $f|_{U_i}$ is a homeomorphism
onto its image. In particular, $f|_{U_i}\colon U_i\to f(U_i)$ is a
quasiconformal mapping in the analytic sense. Since $U$ and $\mathbb{H}^n$ have locally $Q$-bounded
geometry~\cite{hk98}, it follows by~\cite[Theorem 9.8]{hkst01} that $f$ is locally quasisymmetric in $U_i$. By~\cite[Theorem
7.11]{hk98} we then have that $J_{f|_{U_i}}\in L^p_{loc}(U_i)$ for some $p>Q$. 

The above implies
that for each $x\in U\backslash f^{-1}\big(f(\mathcal{B}_f)\big)$, there
exists a relatively compact neighborhood $\mathcal{K}_x$ of $x$ such that
$J_f\in L^p(\mathcal{K}_x)$. On the other hand, by~\cite[Corollary 1.3]{gw15} (or by~\cite[Theorem B]{gnw15}), we have that
 $\mathcal{H}^Q(\mathcal{B}_f)=\mathcal{H}^Q\big(f(\mathcal{B}_f)\big)=0$. Consequently, we have that $J_f\in L^p_{loc}(U)$. Since $f$ is analytically quasiregular, this implies that $f\in N^{1,p}_{loc}(U,\mathbb{H}^n)$. The second claim follows again from the standard Sobolev embedding, see e.g.~\cite[Section 7]{hkst12}.
\end{proof}

We claim first that if $f\colon (U,g)\to (\mathbb{H}^n,g)$ is horizontally $K$-quasiregular, then it is $K^{(n+1)/n}$-quasiregular in the sense of Dairbekov. By Theorem~\ref{thm:equivalence of quasiregular mappings} and Lemma~\ref{lemma:continuity}, we have $f\in N^{1,p}_{loc}(U,\mathbb{H}^n)$ for some $p>Q$. It follows that $f\in HW^{1,Q}_{loc}(U,\mathbb{H}^n)$. Consequently, by~\cite[Theorem 5.1]{d00}, we have that $f$ is $P$-differentiable a.e. in $U$ and the $P$-differential coincides with the formal differential a.e. in $U$. Thus $f$ is contact (see Remark~\ref{rmk:section 4}) and for a.e. $x\in U$ we have
\begin{align*}
	\|D_Hf(x)\|^Q&=\|g^{-1}f^*h(x)\|^{Q/2}\leq K^{(2n+2)/2n}\det(g^{-1}f^*h(x))^{(2n+2)/4n}\\
	&=K^{(n+1)/n}HJ(x,f)^{(n+1)/n}=K^{(n+1)/n}J(x,f).
\end{align*}
Here we have use the observation that it follows from~\eqref{rel_of_jacobs} that
$$
\det(g^{-1}f^*h(x))^{(2n+2)/2n}=\det{(D_Hf(x))}^2=HJ(x,f)^2
$$
since $D_Hf$ here is expressed in the orthonormal frame $(X_1,\ldots,X_{2n})$, in which case $g$ and $h$ are identical matrices.

Conversely, we show that if $f\colon U\to \HH^n$ is $K$-quasiregular in the sense of Dairbekov, then it is horizontally $K^{n/(n+1)}$-quasiregular. Indeed, by~\cite[Theorem 5.1, Theorem 5.5]{d00} $f$ is a branched covering,  it is $P$-differentiable a.e. in $U$ and the $P$-differential coincides with the formal differential a.e. in $U$. Thus for a.e. $x\in U$
\begin{align*}
	\|g^{-1}f^*h(x)\|^k&=\|D_Hf(x)\|^{2k}\leq K^{2n/(n+1)}J(x,f)^{2n/(n+1)}\\
	&=K^{2n/(n+1)}HJ(x,f)^2=K^{2n/(n+1)}\det(g^{-1}f^*h(x)).
\end{align*}
We have shown that the horizontal definition of quasiregularity on Heisenberg groups is quantitatively equivalent with the definition of quasiregularity introduced by Dairbekov~\cite{d99,d00}.

We make an additional observation regarding the first Heisenberg group $\mathbb{H}^1$. In this case $Q=4$ and by Remark~\ref{rmk:step 2}, we have
\begin{align*}
	\det(\ol{g}^{-1}\ol{f^*h}(x))=\lambda_1\lambda_2\mu_1=(\lambda_1\lambda_2)^2=J(x,f)
\end{align*}
for a.e. $x\in U$, where $\lambda_1$ and $\lambda_2$ are the eigenvalues of $g^{-1}f^*h$ and $\mu_1$ is the other eigenvalue of $\ol{g}^{-1}\ol{f^*h}$. Consequently, the definition of \emph{subRiemannian} $K$-quasiregularity coincides with the definition of Dairbekov in this case. It is natural to ask how closely related the subRiemannian definition and Dairbekov's definition are for $n\geq 2$. We leave this question open.

\section{Appendix: equivalence of 1-quasiregular mappings between Riemannian manifolds}\label{sec:Appendix}
In this appendix we prove the equivalence of different notions of $1$-quasiregularity for mappings between Riemannian manifolds. The results we present here are special cases of our more general result, especially of Corollary~\ref{coro:equivalence of 1-quasiregular mappings}, in the subRiemannian setting.

We find that giving an independent presentation in the Riemannian setting has a few benefits: The proof of the equivalence in the Riemannian setting is much more straightforward than that in the subRiemannian case. An independent proof in the Riemannian settings serves better readers that are not familiar with subRiemannian geometry. The presentation here motivates much of the work done in this paper, and in particular it motivates the proof of our main result, Theorem~\ref{thm:equivalence of quasiregular mappings}. 

In this section $(M,g)$ and $(N,h)$ are Riemannian $C^\infty$-smooth Riemannian manifolds of dimension $n$. The equivalence result is the following.

\begin{theorem}\label{thm:equivalence}
	Let $f\colon (M,g)\to (N,h)$ be a non-constant continuous mapping between two Riemannian manifolds. Then the following conditions are equivalent:
	\begin{align*}
	1)\quad f &\mbox{ is metrically $1$-quasiregular}, \\
	2)\quad f &\mbox{ is weak metrically $1$-quasiregular}, \\
	3)\quad f &\mbox{ is analytically $1$-quasiregular}, \\
	4)\quad f &\mbox{ Riemannian $1$-quasiregular}, \\
	5)\quad f &\mbox{ is geometrically $1$-quasiregular}.
	\end{align*}
\end{theorem}
The various notions of quasiregularity in the theorem above can be found from the references~\cite{l14,hp07,hkst01,w12}, or alternatively, they can be understood as special cases of those in the equiregular subRiemannian setting we have introduced earlier. We remark that it follows from this theorem and results in~\cite{ls14} that $1$-quasiregular mappings on $C^\infty$ Riemannian manifolds in any of the above sense are $C^\infty$-smooth. (The works~\cite{ls14, jls} consider also regularity of $1$-quasiregular mappings in the situation when the Riemannian metrics are only in the H\"older class $C^r$, $r>1$. Here we only consider the $C^\infty$ case and refer to those works for more refined regularity results. However, we caution the reader that one has to be careful if trying apply the above theorem in the $C^r$, $r>0$ case. The reason is that our proofs below use for example normal coordinates, which can have unexpected behaviour at least in the case of $C^r$, $0<r<1$, regular metrics.)

We do not impose the topological assumption that $f$ is a branched covering in either the analytic or in the Riemannian definitions of quasiregularity. This is because a non-constant analytically quasiregular mapping is automatically both discrete and open (and thus a branched covering) by a deep theorem of Reshetnyak~\cite{re89}. See also Proposition~\ref{prop:basic properties of quasiregular mappings} below. 

\subsection{Auxiliary results}

The following proposition contains some basic properties of quasiregular mappings between Riemannian manifolds. 

\begin{proposition}[Section 3, \cite{hp07}]\label{prop:basic properties of quasiregular mappings}
	Let $f\colon M\to N$ be a non-constant quasiregular mapping. Then
	
	1) $f$ is sense-preserving, discrete and open,
	
	2) $\Vol_N(f(E))=0$ if and only if $\Vol_M(E)=0$,
	
	3) $\Vol_M(\mathcal{B}_f)=0$,
	
	4) The area formula holds: For all measurable functions $h:N\to [0,\infty]$ and for every measurable set $A\subset M$ there holds that
	\begin{align}\label{eq:area formula}
	\int_{A}h(f(x))J_f(x)d\Vol_M(x)=\int_{N}h(y)N(y,f,A)d\Vol_N(y),
	\end{align}
	where $N(y,f,A)=\card \big(f^{-1}(y)\cap A\big)$ is the multiplicity function of $f$ on $A$.
	
	5) $f$ is differentiable a.e. in $M$.
\end{proposition}

The following lemma is a special case of Theorem~\ref{thm:Differentiability Stepanov}. Since we find the lemma has some value on its own (and the proof of Theorem~\ref{thm:Differentiability Stepanov} is complicated), we give a simple proof of it below.
\begin{lemma}\label{lemma:differentiability}
	Let $f\colon (M,g)\rightarrow (N,h)$ be a continuous mapping that is differentiable at $x_0\in M$. Then there exists a constant $\delta>0$ such that if $d(x,x_0)\leq \delta$, then we have
	$$
	d_h(f(x),f(x_0))=h_{f(x_0)}(f_*V,f_*V)^{1/2} + o\big(d_g(x,x_0)\big),
	$$
	where
	$V=\exp_{x_0}^{-1}(x)\in T_{x_0} M$ with $|V|_{g(x_0)}=d_g(x,x_0)$.
\end{lemma}

\begin{proof}
	Let $(x^i)$ and $(y^{a})$, $i,a=1,\ldots n$, be $g$- and $h$-normal coordinates centered at $x_0$ and $f(x_0)$ respectively. Let $x\in M$ with $d_g(x,x_0)=r$, $r>0$ small. In the above coordinates we have
	$$
	d_h(f(x),f(x_0))=|f(x)|_{h(f(x_0))}=|f(x)|_{\R^n}
	$$
	and
	$$
	d_g(x,x_0)=|x|_{g(x_0)}=|x|_{\R^n}.
	$$
	By the differentiability of $f$ at $x_0$, we have
	$$
	f^i(x)=0+\partial_{a}f^i(0) x^{a}+o(|x|).
	$$
	Combining these, we have
	$$
	d_h(f(x),f(x_0))=|(\partial_af^i(0) x^a)|_{h(f(x_0))}+o(|x|)=|f_* V|_{h(f(x_0))}+o\big(d_g(x,x_0)\big),
	$$
	with 
	$V=\exp_{x_0}^{-1}(x)\in T_{x_0} M$ with $|V|_{g(x_0)}=d_g(x,x_0)$.
\end{proof}

\subsection{Equivalence of Riemannian and analytic 1-quasiregularity}\label{sec:Equivalence of Riemannian and analytic 1-quasiregularity}

It follows from~\cite[Section 5]{w12proc} that if $f\colon M\to N$ is a non-constant quasiregular mapping between two Riemannian manifolds, then $\Lip f$ is the minimal $n$-weak upper gradient of $f$. We prove a couple of lemmas will yield the equivalence of Riemannian and analytic 1-quasiregularity. The first lemma shows that at the points $x_0$ of differentiability, $\Lip f(x_0)$ coincides with the supremum norm 
$$\|Df(x_0)\|:=\sup_{V\in T_{x_0}M}\big\{|Df(x_0)V|_{h(f(x_0))}:|V|_{g(x_0)}\leq 1\big\}.$$


\begin{lemma}\label{lemma:weak upper gradient coincide with super norm}
	Let $f\colon (M,g)\to (N,h)$ be a continuous mapping. If $f$ is differentiable at $x_0\in M$, then 
	\begin{align}\label{eq:Lip = sup norm}
	\Lip f(x_0)=\|Df(x_0)\|.
	\end{align}
\end{lemma}
\begin{proof}
	This is an immediate consequence of Lemma~\ref{lemma:differentiability}. Indeed, 
	\begin{align*}
	\Lip f(x_0)&=\limsup_{r\to 0}\sup_{x\in B(x_0,r)}\frac{d_N(f(x),f(x_0))}{r} =\limsup_{r\to 0}\sup_{|V|_{g(x_0)}\leq r}\frac{|Df(x_0)V|_{h(f(x_0))}}{r} \\
	&=\limsup_{r\to 0}\sup_{|V|_{g(x_0)}\leq r}\frac{|V|_{g(x_0)}}{r}\frac{|Df(x_0)V|_{h_f(x_0)}}{|V|_{g(x_0)}}=\|Df(x_0)\|.
	\end{align*}
\end{proof}

The next result is an analogue of Lemma~\ref{lemma:coincidence of two Jacobian} and in our case, it follows directly from the change of variable formula.
\begin{lemma}\label{lemma:determinant coincides with volume derivative}
	Let $f\colon (M,g)\to (N,h)$ be a quasiregular mapping. Then
	\begin{align}\label{eq:determinant = volume derivative}
	\Dett( Df(x_0))=J_f(x_0)
	\end{align}
	for a.e. $x_0\in M$.
\end{lemma}
\begin{proof}
	By Proposition~\ref{prop:basic properties of quasiregular mappings} 4) and~\cite[Theorem 2.2.4]{l14},
	\begin{align}\label{eq:det = vol}
	\int_AJ_f(x)d\Vol_M(x)=\int_A\Dett(Df(x))d\Vol_M(x)
	\end{align}
	for all measurable set $A\subset M$. Note that $J_f$ and $\Dett(Df)$ are locally integrable on $M$, the claim follows from~\eqref{eq:det = vol} and the Lebesgue differentiation theorem.
\end{proof}

Plugging the results of these lemmas into the definitions of analytic $1$-quasiregularity and Riemannian $1$-quasiregularity shows that they are equivalent.

\subsection{Equivalence of Riemannian and metric 1-quasiregularity}\label{sec:Equivalence of Riemannian and metric 1-quasiregularity}

We show that Riemannian $1$-quasiconformal mappings and $1$-quasiconformal mappings defined in the (weak) metric sense are equivalent. 

Let us first show Riemannian $1$-quasiregular mappings are $1$-quasiregular in the metric sense. 
\begin{proposition}\label{prop:Riemannian implies metric}
	Let $f\colon (M,g)\rightarrow (N,h)$ be a Riemannian $1$-quasiregular mapping. Then it is metrically $1$-quasiregular.
\end{proposition}
\begin{proof}
	By~\cite[Theorem 4.4]{ls14}, Riemannian $1$-quasiregular mappings are $C^2$ regular and they satisfy the equation of conformality
	$$
	f^*h=cg.
	$$
	By lemma~\ref{lemma:differentiability} we have
	\begin{multline}
		\sup_{y\in \partial B(x,r)} d(f(x),f(y))=\sup_{y\in \partial{B}(x,r)}\big(h_{f(x)}(f_*V,f_*V\big)^{1/2} + o(d_g(x,y))\big)\\
		=\sup_{y\in \partial{B}(x,r)}\big(c(x)g_x(V,V)^{1/2} + o(d_g(x,y))\big)=c(x)r +o(r)
	\end{multline}
	and similarly
	$$
	\inf_{y\in \partial B(x,r)} d(f(x),f(y))= c(x)r+o(r).
	$$
	Thus
	$$
	\limsup_{r\rightarrow 0} \frac{\sup_{y\in \partial{B}(x,r)} d(f(x),f(y))}{\inf_{y\in\partial B(x,r)} d(f(x),f(y))}\leq 1.
	$$
\end{proof}

In the other direction we have: 

\begin{proposition}\label{prop:metric implies Riemannian}
	Let $f\colon (M,g)\rightarrow (N,h)$ be a weak metrically $1$-quasiregular mapping. Then $f$ is Riemannian $1$-quasiregular.
\end{proposition}
\begin{proof}
	By Proposition~\ref{prop:basic properties of quasiregular mappings} we have that $f$ is differentiable a.e. in $M$. In the same vein as above, fix a point $x\in M$ such that $f$ is differentiable at $x$. Then we have
	\begin{align}
		\sup_{y\in \bar{B}(x,r)} d(f(x),f(y))&=\sup_{y\in \bar{B}(x,r)}\big(h_{f(x)}(f_*V,f_*V)^{1/2} + o(d_g(x,y))\big) \\
		&=\sup_{|V|_{g(x)}\leq r}|Df V|_{h(f(x))} + o(r) 
	\end{align}
	and
	$$
	\inf_{y\in \partial B(x,r)} d(f(x),f(y))=\inf_{|V|_{g(x)}= r}|Df V|_{h(f(x))} + o(r).
	$$
	Here we have used the fact that $V=\exp_x^{-1}y$, $|V|_{g(x)}=d_g(x,y)$. It follows that
	$$
	1=\liminf_{r\rightarrow 0} \frac{\sup_{y\in \bar{B}(x,r)} d(f(x),f(y))}{\inf_{y\in M\backslash B(x,r)} d(f(x),f(y))}=\frac{\lambda_{max}(x)}{\lambda_{min}(x)},
	$$
	where $\lambda_{max}(x)$ and $\lambda_{min}(x)$ are the largest and smallest eigenvalue of the matrix $g^{-1}f^*h$.
	Since the eigenvalues of $g^{-1}f^*h$ are all the same, we have
	$$
	g^{-1}f^*h = c I_{n\times n}.
	$$
	Since this holds at a.e. every point, $f$ is a conformal mapping
	$$
	f^*h=cg,
	$$
	for some a.e. positive function $c$ and thus Riemannian $1$-quasiregular by~\cite[Theorem 4.4]{ls14}.
\end{proof}

\subsubsection*{Acknowledgements}
C.-Y. Guo was supported by the Magnus Ehrnrooth foundation, the Finnish Cultural Foundation--Central Finland Regional Fund (No.~30151735) and the Swiss National Science Foundation (No.~153599). T. Liimatainen was partly supported by the Academy of Finland (Centre of Excellence in Inverse Problems Research) and an ERC Starting Grant (grant agreement no 307023).

C.-Y. Guo would like to thank Prof.~Marshall Williams for many useful discussions on the definitions of quasiregularity in general metric measure spaces, as well as for many other discussion on topics related to analysis on metric spaces during the past years. 

The authors are indebted to Prof.~Enrico Le Donne for his insightful comments, in particular, for pointing out a mistake in the preliminary version of the manuscript. They would like to thank Profs.~Pekka Koskela,~Pekka Pankka and~Ilkka Holopainen for helpful comments on Corollary~\ref{coro:equivalence of 1-quasiregular mappings} in the Riemannian manifolds, and Profs.~Jeremy Tyson,~Kirsi Peltonen, and~Katrin F\"assler for their interest in this work.

\end{document}